\documentclass[10pt,letterpaper]{amsart}

\usepackage{amsmath,amsthm,amsfonts,amssymb}
\usepackage{mathtools} %extension of amsmath, provides bug fixes and other tools 
\usepackage{mathrsfs} %Provides a nice \mathscr command
\usepackage{stmaryrd}
\usepackage{cancel}
\usepackage{multirow}
\usepackage{longtable}
\usepackage[thinlines]{easybmat}
\usepackage[pdftex,bookmarks=true]{hyperref} %adds hypertext links, can't use in ``draft'' mode, and must use in pdf mode
\usepackage{graphicx}
\usepackage{amsrefs}

\newcommand{\mfr}[1]{\ensuremath \mathfrak{#1}}
\newcommand{\mbb}[1]{\ensuremath \mathbb{#1}}
\newcommand{\mbf}[1]{\ensuremath \mathbf{#1}}
\newcommand{\mcl}[1]{\ensuremath \mathcal{#1}}

\newcommand{\mrm}[1]{\ensuremath \mathrm{#1}}

\newcommand{\R}{\mathbb{R}}

\DeclareMathOperator{\GL}{GL}

\DeclareMathOperator{\tr}{tr}
\DeclareMathOperator{\Gr}{Gr}
\DeclareMathOperator{\Rm}{Rm}
\DeclareMathOperator{\Ric}{Ric}

\DeclareMathOperator{\Rc}{Rc}

\DeclareMathOperator{\scal}{scal}
\DeclareMathOperator{\Lie}{Lie}

\DeclareMathOperator{\ad}{ad}
\DeclareMathOperator{\id}{id}
\DeclareMathOperator{\Der}{Der}
\DeclareMathOperator{\End}{End}
\DeclareMathOperator{\sym}{sym}
\DeclareMathOperator{\rank}{rank}
\DeclareMathOperator{\Aut}{Aut}

\newcommand{\Ro}{\mathring{R}}
\newcommand{\vv}{\vphantom{${\displaystyle \sum}$}}
\newcommand{\ip}[1]{\langle #1 \rangle}
\newcommand{\opA}{\mbf{L}}

\newcommand{\smfrac}[2]{{\textstyle{\frac{#1}{#2}}}}
\newcommand{\half}{{\smfrac{1}{2}}}

\numberwithin{equation}{section}

\theoremstyle{plain} %% This is the default, anyway
\newtheorem{thm}[equation]{Theorem}
\newtheorem{cor}[equation]{Corollary}
\newtheorem{lem}[equation]{Lemma}
\newtheorem{prop}[equation]{Proposition}

\theoremstyle{definition}

\theoremstyle{remark}
\newtheorem{rem}[equation]{Remark}
\newtheorem{ex}[equation]{Example}

\begin{document}

\author{Michael Jablonski}
\address{Department of Mathematics, University of Oklahoma}
\email{mjablonski@math.ou.edu}
\urladdr{http://www2.math.ou.edu/~mjablonski/}

\author{Peter Petersen}
\address{Department of Mathematics, University of California, Los Angeles}
\email{petersen@math.ucla.edu}
\urladdr{http://www.math.ucla.edu/~petersen/}

\author[M.~B.~Williams]{Michael Bradford Williams}
\address{Department of Mathematics, University of California, Los Angeles}
\email{mwilliams@math.ucla.edu}
\urladdr{http://www.math.ucla.edu/~mwilliams/}

\thanks{The first author was supported by NSF grant DMS-1105647 and the second author was supported by NSF grant DMS-1006677.}

\title[Linear stability of algebraic {R}icci solitons]{Linear stability of algebraic {R}icci solitons}
%\date{\today}

\begin{abstract}
We consider a modified Ricci flow equation whose stationary solutions include Einstein and Ricci soliton metrics, and we study the linear stability of those solutions relative to the flow.  After deriving various criteria that imply linear stability, we turn our attention to left-invariant soliton metrics on (non-compact) simply connected solvable Lie groups and prove linear stability of many such metrics.  These include an open set of two-step solvsolitons, all two-step nilsolitons, two infinite families of three-step solvable Einstein metrics, all nilsolitons of dimensions six or less, and all solvable Einstein metrics of dimension seven or less with codimension-one nilradical.  For each linearly stable metric, dynamical stability follows from a generalization of the techniques of Guenther, Isenberg, and Knopf.
\end{abstract}

%%% AMS subject classification
%53C25    	Special Riemannian manifolds (Einstein, Sasakian, etc.)
%53C30     	Homogeneous manifolds
%53C44    	Geometric evolution equations (mean curvature flow, Ricci flow, etc.)
%22E25    	Nilpotent and solvable Lie groups

%%%-------------------------------------------------------------------
%%% body of the paper
%%%-------------------------------------------------------------------

\maketitle

\tableofcontents
%\listoftables
%\listoffigures

\section{Introduction}\label{sec:intro}

Ricci flow is now a standard tool in geometry and topology, and one of its main uses is to provide a means of deforming a given Riemannian metric, hopefully into one that is ``distinguished'' in some sense.  On a fixed Riemannian manifold, Ricci flow may be viewed as a dynamical system on the space of Riemannian metrics, modulo diffeomorphisms.  In this context, distinguished metrics are Einstein or Ricci soliton metrics, and we phrase the problem of ``deforming into a distinguished metric'' in terms of dynamical stability.  Namely, given a stationary solution $g_0$ of Ricci flow and some topology on the space of metrics, does there exist a neighborhood $U$ of $g_0$ such that all Ricci flow solutions with initial data in $U$ converge to $g_0$?

Many authors have considered stability for solutions of Ricci flow; see the introduction of \cite{Wu2013} for a description of some of the results and techniques used.  The approach to stability that we take in this paper follows the program initiated by Guenther, Isenberg, and Knopf in \cite{GuentherIsenbergKnopf2002}, which, roughly speaking, has three steps:
\begin{enumerate}
\item[1.] Modify Ricci flow (e.g., by diffeomorphisms and rescaling) so that it has a suitable class of fixed points.
\item[2.] Compute the linearization at a fixed point and prove linear stability.
\item[3.] Set up a sequence of tensor spaces with certain interpolation properties and apply a theorem of Simonett to get dynamical stability.
\end{enumerate}
Since its introduction in \cite{GuentherIsenbergKnopf2002}, this program has been used to prove a variety of other stability results \cite{Knopf2009,KnopfYoung2009,Young2010,Williams2013-systems,Wu2013}.

For Step 1, the flow that we wish to study appeared in \cite{GuentherIsenbergKnopf2006}.  Given a manifold $\mcl{M}^n$, $\lambda \in \R$, and a vector field $X$, the \textit{curvature-normalized Ricci flow} is
\begin{equation}\label{eq:cnrf}
\partial_t g = -2 \Rc + 2\lambda g + \mcl{L}_X g.
\end{equation}
A Ricci soliton $(\mcl{M},g,\lambda,X)$, which satisfies
\begin{equation}\label{eq:ricci-soliton}
\Rc = \lambda g + \half \mcl{L}_X g,
\end{equation}
is clearly a stationary solution of \eqref{eq:cnrf}.  (Recall that an Einstein metric corresponds to a Ricci soliton where the field $X$ is a Killing field.)

For Step 2, the linearization of the flow \eqref{eq:cnrf} was also considered in \cite{GuentherIsenbergKnopf2006}.  After modification by DeTurck diffeomorphisms, the linearization is
\begin{equation}\label{eq:cnrf-lin}
\partial_t h = \opA h := \Delta_L h + 2\lambda h + \mcl{L}_X h,
\end{equation}
where $\Delta_L h$ is the Lichnerowicz Laplacian acting on symmetric 2-tensors.  A stationary solution of \eqref{eq:cnrf} is \textit{strictly} (resp.~\textit{weakly}) \textit{linearly stable} if the operator $\opA$ has negative (resp.~non-positive) spectrum, that is, if there exists some $\epsilon >0$ (resp.~$\epsilon=0$) such that
\[ (\opA h,h) \leq - \epsilon \|h\|^2 \]
for all symmetric 2-tensors $h$ taken from some appropriate tensor space.  

We emphasize that the notion of linear stability for a given metric depends on a particular linear operator.  For example, Koiso studied stability of Einstein metrics with respect to the second variation of the total scalar curvature functional \cite{Koiso1979,Koiso1980,Koiso1982}.  Other authors have studied stability of shrinking ($\lambda > 0$), gradient ($X = \nabla f$ for some $f \in C^\infty(\mcl{M})$) Ricci solitons with respect to the second variation of Perelman's $\nu$-functional; see, for example \cite{HallMurphy2011,CaoZhu2012,CaoHe2013}.  

In this paper, we will not be directly concerned with Step 3 and the aspects of dynamical stability.  Indeed, our goal is to study the \textit{linear stability} of homogeneous Einstein and Ricci soliton metrics on non-compact manifolds, with respect to the operator from \eqref{eq:cnrf-lin}.  (Step 3 is, however, the primary focus of a companion paper by Wu and the third author \cite{WilliamsWu2013-dynamical}, see Theorem \ref{thm:dynamical-all} below.)  The first results for such metrics are found in \cite{GuentherIsenbergKnopf2006}, where the authors give a detailed analysis of homogeneous solitons on three spaces:~$\mrm{Nil}^3$, $\mrm{Sol}^3$, and $\mrm{Nil}^4$.  

Since that work, much has been discovered about homogeneous solitons.  It turns out that the only non-trivial homogeneous Ricci solitons occur in the non-compact, expanding ($\lambda < 0$), and non-gradient ($X \ne \nabla f$ for any $f \in C^\infty(\mcl{M})$) case; see Section 2 of \cite{Lauret2011-sol} for a discussion of this fact.  Additionally, all known examples can be realized as left-invariant metrics on simply connected solvable Lie groups (or, equivalently, as inner products on solvable Lie algebras---we'll use both interpretations), such that
\begin{equation}\label{eq:alg-soliton}
\Ric = \lambda \id + D,
\end{equation}
where $\Ric$ is the Ricci endomorphism, $\lambda \in \mbb{R}$, and $D$ is a derivation of the Lie algebra.  A metric satisfying this equation is called an \textit{algebraic soliton}; on nilpotent and solvable Lie groups, such metrics are called \textit{nilsolitons} and \textit{solvsolitons}, respectively.  

An algebraic soliton is in fact a Ricci soliton (see, e.g., \cite{Lauret2011-sol}); conversely, a Ricci soliton on a solvable Lie group is isometric to a solvsoliton (on a possibly different Lie group).  Also, a solvable Lie group admitting a non-flat algebraic soliton is diffeomorphic to $\mbb{R}^n$.  These last two facts appear in \cite{Jablonski2011-hom}.  Lauret has also proven strong structural results for algebraic solitons \cite{Lauret2011-sol}.  Essentially, nilsolitons are the nilpotent parts of solvsolitons---a solvsoliton implies the existence of a nilsoliton on the nilradical, and conversely, any nilsoliton can be extended in a specific way to a solvsoliton.  See Theorem \ref{thm:lauret} for more details.  This new knowledge allows us to expand upon the linear stability work done in \cite{GuentherIsenbergKnopf2006}.

Here is an outline of the paper and a description of the results.  In Section \ref{sec:lin-criteria}, we analyze the linearized Ricci flow operator from \eqref{eq:cnrf-lin} and derive a sufficient condition for linear stability of soliton metrics (including Einstein metrics) under the flow.  This condition covers a somewhat general class of metrics, but our main focus is the algebraic soliton case, where $\mrm{div}(X) = \tr D$.  

\begin{prop}\label{prop:stab-conds}
Let $(\mcl{M},g,\lambda,X)$ be a Ricci soliton with constant scalar curvature.  The following condition implies that $g$ is strictly linearly stable:
\begin{equation}\label{eq:stab-cond}
(\Ro h + \Ric \circ h,h) < \half \mrm{div}(X) \|h\|^2
\end{equation}
for all symmetric 2-tensors $h$.  When $g$ is $\lambda$-Einstein, this condition is
\begin{equation}\label{eq:stab-einst}
(\Ro h,h) < -\lambda \|h\|^2.
\end{equation}
\end{prop}

Here, $\Ro$ is the action of the Riemann curvature tensor on symmetric 2-tensors.  We also have criteria based on sectional curvature bounds.

\begin{prop}\label{prop:sec-stab}
Let $(\mcl{M},g,\lambda,X)$ be a Ricci soliton with constant scalar curvature, and suppose that the sectional curvature satisfies $\sec \leq K \leq 0$ for some constant $K$.  A sufficient condition for $g$ to be strictly linearly stable is that
\[ (n-2) K < \half \mrm{div}(X). \]
\end{prop}

We emphasize that these results only give \textit{sufficient} conditions for linear stability, as they do not involve the full operator from \eqref{eq:cnrf-lin}.  The hope is that, in the case of solvable Lie groups, one can use algebraic methods to verify these conditions relatively simply.  It turns out that this hope is too optimistic, as there are examples of metrics---both Einstein and non-Einstein---for which these Propositions fail to give stability; see Subsection \ref{subsec:abelian-einstein}.  However, as we will see in Theorem \ref{thm:stab-all}, these Propositions succeed in many other cases.

In Section \ref{sec:extensions}, we explore the relationship between linear stability of a solvable Einstein metric and its associated nilsoliton in the case of codimension-one nilradical.  Using Proposition \ref{prop:stab-conds}, we show that stability of the Einstein metric implies stability of the nilsoliton when the eigenvalues of the nilsoliton derivation are not too large, relative to the trace of the derivation.

\begin{prop}\label{prop:ext-stab}
Let $\mfr{s}$ be a solvable Lie algebra with codimension-one nilradical $\mfr{n} = [\mfr{s},\mfr{s}]$.  If a $\lambda$-Einstein metric on $\mfr{s}$ satisfies \eqref{eq:stab-einst}, then a sufficient condition for the nilsoliton on $\mfr{n}$ to be strictly linearly stable is that
\[ \max\{d_i\} < \frac{1}{2+\sqrt{2}} \tr D, \]
where $d_1,\dots,d_n$ are the eigenvalues of the nilsoliton derivation $D$.
\end{prop}

Also in that section we use the continuity of the curvature operators to show that any non-nilpotent solvsoliton satisfying \eqref{eq:stab-cond} is contained in an open set of stable solvsolitons.  This set is open in the moduli space of solvsoliton extensions of the nilradical of the original stable metric.

In Section \ref{sec:abelian}, we consider abelian nilsolitons (which are trivially stable) and their solvsoliton extensions.  We calculate the sectional curvatures and use Proposition \ref{prop:sec-stab} to find an open set (in an appropriate moduli space) of stable solvsolitons whose nilradicals have codimension one; see Proposition \ref{prop:abelian-sec}.

In Section \ref{sec:2step}, we use Proposition \ref{prop:stab-conds} to derive conditions for stability of two-step nilsolitons.  Using structural results due to Eberlein \cite{Eberlein2008}, and analyzing the possible Ricci tensors of nilsoliton metrics, we prove that all two-step nilsolitons are stable; see Proposition \ref{prop:2step-stab}.  The space of all such solitons is quite large; in particular, there are continuous families \cite{Jablonski2011-moduli}.

In Section \ref{sec:3step}, we consider linear stability of solvable Einstein metrics whose nilradicals are two-step and codimension-one.  While the techniques of Section \ref{sec:2step} do not work in all cases here, we are able to obtain stability for two countably infinite families of metrics.  First are those where the nilradical is a generalized Heisenberg group, and second are where the nilradical is a free two-step nilpotent Lie group; see Proposition \ref{prop:stab-solv-heis-free}.  We also briefly discuss spaces of negative curvature and Proposition \ref{prop:sec-stab}.  In particular, we use results of Heber \cite{Heber1998} to find high-dimensional families of stable Einstein metrics that arise as deformations of certain quaternionic and Cayley hyperbolic spaces.

In Section \ref{sec:low-dim}, we consider ``low-dimensional'' examples.  Namely, we prove linear stability of all nilsolitons in dimension six or less, using the classification found in \cite{Will2011}.  We also prove linear stability for the corresponding one-dimensional solvable Einstein extensions, except in the case of the hyperbolic space $H^2$, which is only weakly linearly stable.  This gives almost 100 examples of stable metrics on solvable Lie groups.  See Proposition \ref{prop:low-dim} and Tables \ref{table:solitons} and \ref{table:solitons2}.  We show that each metric in a curve of seven-dimensional nilsolitons is strictly linearly stable, and also their one-dimensional Einstein extensions.  Finally, we show that each solvsoliton extension of the three-dimensional Heisenberg algebra is strictly linearly stable. 

Here is a summary of the spaces shown to be stable.

\begin{thm}\label{thm:stab-all}
The following algebraic solitons are strictly linearly stable with respect to the curvature-normalized Ricci flow:
\begin{enumerate}
\item\label{item1} every nilsoliton of dimension six or less, and every member of a certain one-parameter family of seven-dimensional nilsolitons;
\item every abelian or two-step nilsoliton;
\item every four-dimensional solvsoliton whose nilradical is the three-dimensional Heisenberg algebra;
\item an open set of solvsolitons whose nilradicals are codimension-one and abelian;
\item every solvable Einstein metric whose nilradical is codimension-one and 
\begin{enumerate}
\item found in \ref{item1} and has dimension greater than one, or
\item a generalized Heisenberg algebra, or
\item a free two-step nilpotent algebra;
\end{enumerate}
\item for each $m \geq 2$, an $(8m^2-6m-8)$-dimensional family of negatively-curved Einstein metrics containing the quaternionic hyperbolic space $\mbb{H}H^{m+1}$;
\item an $84$-dimensional family of negatively-curved Einstein metrics containing the Cayley hyperbolic plane $\mbb{C}aH^2$.
\end{enumerate}
Further, any non-nilpotent solvsoliton $\mfr{s} = \mfr{n} \rtimes \mfr{a}$ satisfying \eqref{eq:stab-cond} and $0 < \dim(\mfr{a}) < \rank(\mfr{n})$ is contained in an open set of stable solvsolitons.
\end{thm}

This list is, of course, not complete.  One might be tempted to conjecture that all solvsolitons are strictly linearly stable, but this is not clear.  As mentioned above, in Subsection \ref{subsec:abelian-einstein} we give examples of metrics for which both Propositions \ref{prop:stab-conds} and \ref{prop:sec-stab} fail to give stability.  These metrics are possibly stable, but our methods are insufficient to prove it.

As the focus of this paper is linear stability, we will not directly consider the dynamical aspects of stability.  However, using the results of \cite{WilliamsWu2013-dynamical}, we have dynamical stability for each metric in Theorem \ref{thm:stab-all}.

\begin{thm}[\cite{WilliamsWu2013-dynamical}]\label{thm:dynamical-all}
Each metric from Theorem \ref{thm:stab-all} is dynamically stable in the context of \cite[Theorem 1.1]{WilliamsWu2013-dynamical} for Einstein metrics and of \cite[Theorem 1.2]{WilliamsWu2013-dynamical} for non-Einstein solitons.
\end{thm}

This theorem adds to the understanding of long-time behavior of Ricci flow solutions.  Indeed, expanding homogeneous solitons are expected to act as singularity models for (at least some) Type-III Ricci flow singularities, and this is known to be true in some cases \cite{Lott2010, Knopf2009}.  Theorem \ref{thm:dynamical-all} exhibits many more examples of Ricci flow solutions that converge to expanding homogeneous solitons.

\section{Criteria for linear stability}\label{sec:lin-criteria}

In this section we find conditions that imply linear stability, and doing so requires us to understand the linear operator $\opA$ from \eqref{eq:cnrf-lin}.  In particular, we analyze the various components that make up the $L^2$ inner product $(\opA h,h)$, which include various curvature operators, Laplace operators, and Lie derivatives.  The main results of this section, Propositions \ref{prop:stab-conds} and \ref{prop:sec-stab}, give sufficient conditions for linear stability in the case of constant scalar curvature.  The first requires estimating the eigenvalues of $\Rm$, acting on symmetric 2-tensors, and the second requires sectional curvature bounds.
The results are also adapted to the case of algebraic solitons.

\subsection{Curvature operators}
  
Let $(\mcl{M},g)$ be a Riemannian manifold.  The Lichnerowicz Laplacian, acting on symmetric 2-tensors, has the form
\begin{equation}\label{eq:lich1}
-\Delta_L h = \nabla^*\nabla h - 2 \Ro h + \Ric \circ h + h \circ \Ric.
\end{equation}
Here, $\nabla^*\nabla h = -\tr_g \nabla^2 h = -\Delta h$ is the connection Laplacian on 2-tensors.  

The second term is the action of the Riemann tensor on symmetric 2-tensors.  In local coordinates, this is
\begin{equation}\label{eq:Ro-coords}
(\Ro h)_{ij} = R_{ipqj} h^{pq}.
\end{equation}
This operator satisfies $\Ro g = \Rc$, so $\langle \Ro g,g \rangle = \scal$.  When $g$ is $\lambda$-Einstein, $\Ro g = \lambda g$, so $\lambda$ is an eigenvalue of $\Ro$.  Additionally, when $g$ is $\lambda$-Einstein, we have $\tr (\Ro h) = \lambda \tr h$, so that it is actually an operator on trace-free symmetric $2$-tensors.  We will not need this fact, however.

The last two terms in \eqref{eq:lich1} are an abuse of notation for the action of the Ricci tensor on symmetric 2-tensors.  In local coordinates, this is
\begin{equation}\label{eq:Rico-coords}
(\Ric \circ h + h \circ \Ric)_{ij} = R_i^k h_{kj} + R_j^k h_{ki}. 
\end{equation}
We can combine these two curvature operators to form $\mfr{Ric}$, the \textit{Weitzenb\"ock curvature operator}, acting on symmetric 2-tensors\footnote{More generally, the Weitzenb\"ock curvature operator acts on $p$-tensors by
\[ \mfr{Ric}(T)(X_1,\dots,X_p) = \sum_{j=1}^n \sum_{k=1}^p \big( R(e_j,X_k) T \big) (X_1,\dots,e_j,\dots,X_p), \]
in an orthonormal frame $\{e_j\}$.  Here, the action of the curvature tensor $R$ on a tensor $T$ is
\[ \big(R(U,V) T\big) (X_1,\dots,X_p) = - \sum_{k=1}^p T\big( X_1,\dots,R(U,V) X_k, \dots, X_p \big). \]
One gets a Lichnerowicz Laplacian on $p$ tensors by $\nabla^* \nabla + \mfr{Ric}$.}:
\begin{align*}
\mfr{Ric}(h) &= -2 \Ro + \Ric \circ h + h \circ \Ric \\
\mfr{Ric}(h)_{ij} &= - 2 R_{ipqj} h^{pq} + R_i^k h_{kj} + R_j^k h_{ki}
\end{align*}
Note that when $g$ is $\lambda$-Einstein, this becomes $\mfr{Ric}(h) = 2(\lambda - \Ro) h$.  This operator allows us to write \eqref{eq:lich1} as
\[ -\Delta_L h = \nabla^* \nabla h + \mfr{Ric}(h). \]

An important formula involving the Lichnerowicz Laplacian is Koiso's Bochner formula \cite{Koiso1979}, which we adapt here in a slightly more general setting.  

\begin{prop}\label{prop:koiso}
On a Riemannian manifold $(\mcl{M},g)$, let $h$ be a symmetric 2-tensor with compact support and define a 3-tensor $T$ by
\[ T_{ijk} := \nabla_k h_{ij} - \nabla_i h_{jk}. \] 
Then
\begin{equation}\label{eq:koiso-gen}
\|\nabla h\|^2 = \half \|T\|^2 + \|\delta h\|^2 - \half (\mfr{Ric}(h),h).
\end{equation}
\end{prop}

We sketch the proof.  First, it is easy to see that
\[ \|T\|^2 = 2 \|\nabla h \|^2 - 2\int \nabla_k h_{ij} \nabla_\ell h_{pq} g^{i\ell} g^{jp} g^{kq}. \]
Integrate by parts and commute covariant derivatives to obtain
\begin{align*}
- 2\int \nabla_k h_{ij} \nabla_\ell h_{pq} g^{i\ell} g^{jp} g^{kq}
&= - 2\|\delta h\|^2  + 2 \int R_i^j h_j^k h_k^i - 2 \int R_{ijkl} h^{i\ell} h^{jk} \\
&= - 2\|\delta h\|^2  + (\mfr{Ric}(h),h).
\end{align*}
Now \eqref{eq:koiso-gen} follows from combining these two equations.

\subsection{Laplacians}

We saw ``the'' Lichnerowicz Laplacian above,
\[ -\Delta_L = \nabla^*\nabla + \mfr{Ric}, \]
 but there is actually a family of Laplacians depending on a parameter $c>0$, 
\[ -\Delta_{L,c} := \nabla^* \nabla + c \, \mfr{Ric}, \]
As a special case, let us consider the \textit{Codazzi Laplacian}:
\[ -\Delta_C := -\Delta_{L,1/2} = \nabla^*\nabla + \half \mfr{Ric}. \]
The name arises from the fact that the kernel of this operator, acting on symmetric 2-tensors, consists of Codazzi tensors.  By \eqref{eq:koiso-gen}, we have
\begin{equation}\label{eq:koiso-codazzi}
(-\Delta_C h,h) = -\half\|T\|^2 - \|\delta h\|^2, 
\end{equation}
so this operator is non-positive.

\subsection{Lie Derivatives}

Consider a Ricci soliton $(\mcl{M},g,\lambda,X)$ satisfying \eqref{eq:ricci-soliton}.  Using the expression of the Lie derivative in terms of covariant derivatives, we have
\[ \half (\mcl{L}_X g)(Y,Z) 
= \half g(\nabla_Y X,Z) + \half g(Y,\nabla_Z X)
= g( DY,Z ), \]
where 
\[ D := \half(\nabla X + (\nabla X)^* ). \]
This gives an equivalent soliton equation on the endomorphism level:
\[ \Ric = \lambda \id_{T\mcl{M}} + D. \]

Express the Lie derivative of a compactly supported symmetric 2-tensor $h$ as 
\[ (\mcl{L}_X h)(Y,Z) = (\nabla_X h)(Y,Z) + \underbrace{h(\nabla_Y X,Z) + h(Y,\nabla_Z X)}_{\Xi_X(h)(Y,Z):=}. \]
Writing $\langle \nabla_X h,h \rangle = \half \nabla_X |h|^2$ and integrating by parts gives
\begin{equation}\label{eq:div-ibp}
\int_M \langle \nabla_X h,h \rangle \, d\mu 
= \half \int \nabla_X |h|^2 \, d\mu
= \half \int \langle \nabla |h|^2, X \rangle \, d\mu
= -\half \int |h|^2 \mrm{div}(X) \, d\mu.
\end{equation}
If $\mrm{div}(X)$ is constant (which happens when $g$ has constant scalar curvature), this becomes
\[ (\nabla_X h,h) = -\half \mrm{div}(X) \|h\|^2. \]

Next, consider $\Xi$ and $h$ on the endomorphism level, so that 
\[ \Xi_X(h) = h \circ \nabla X + (\nabla X)^* \circ h. \]
Using this and the endomorphism inner product $\langle S,T \rangle = \tr(S^*T)$, we have
\begin{align*}
\langle\Xi_X(h),h \rangle
&= \tr(h \circ \nabla X \circ h + (\nabla X)^* \circ h \circ h) \\
&= \tr((\nabla X + (\nabla X)^*) \circ h \circ h) \\
&= 2 \langle D \circ h,h \rangle \\
&= 2 \langle (\Ric - \lambda \id) \circ h,h \rangle \\
&= 2 \langle \Ric \circ h,h \rangle - 2 \lambda |h|^2. 
\end{align*}
Now, combining this with \eqref{eq:div-ibp} gives
\begin{equation}\label{eq:lie-simp}
(\mcl{L}_X h,h) = - \half \mrm{div}(X) \|h\|^2 + 2 (\Ric \circ h,h) - 2\lambda \|h\|^2.
\end{equation}

\subsection{Rewriting the linearization}

We now use the various formulas above to rewrite the linearized operator $\opA$, and describe some sufficient conditions for linear stability.  Define the quadratic form
\[ Q(h) := (\Ro h + \Ric \circ h,h). \]

\begin{prop}
Let $(\mcl{M},g,\lambda,X)$ be a Ricci soliton with constant scalar curvature.  The following condition implies that $g$ is strictly linearly stable:
\[ Q(h) < \half \mrm{div}(X) \|h\|^2. \]
When $g$ is $\lambda$-Einstein, this condition is
\[ (\Ro h,h) < -\lambda \|h\|^2. \]
\end{prop}

\begin{proof}
The linearization is
\begin{align*}
\opA h 
&= \Delta_L h + 2\lambda h + \mcl{L}_X h \\
&= -\nabla^* \nabla h - \mfr{Ric}(h) + 2\lambda h + \mcl{L}_X h \\
&= -\Delta_C h - \half \mfr{Ric}(h) + 2\lambda h + \mcl{L}_X h.
\end{align*}
Using \eqref{eq:koiso-codazzi} and \eqref{eq:lie-simp}, we have
\begin{equation}\label{eq:A-exp-1}
\begin{aligned}
(\opA h,h) 
&= -(\Delta_C h,h) - \half (\mfr{Ric}(h),h) + 2 \lambda \|h\|^2 \\
&\quad  - \half \mrm{div}(X) \|h\|^2 + 2 (\Ric \circ h,h ) - 2\lambda \|h\|^2 \\
&= -\half\|T\|^2 - \|\delta h\|^2 + (\Ro h,h) + (\Ric \circ h,h) - \half \mrm{div}(X) \|h\|^2.
\end{aligned}
\end{equation}
The curvature terms are exactly $Q(h)$, and since we want $(\opA h,h) < 0$, the result follows.
\end{proof}

\begin{rem}
In the Einstein case, the condition for linear stability is the same as the stability condition considered by Koiso and others; see \cite{Koiso1979} or \cite{Besse2008}.
\end{rem}

Recall that a homogeneous metric has constant scalar curvature.  Furthermore, in the case of an algebraic soliton on a simply connected solvable Lie group $\mcl{S}$ with Lie algebra $\mfr{s}$, we have $\Ric = \lambda \, \id + D$, where $D \in \Der(\mfr{s})$.  Tracing this and comparing with the traced version of \eqref{eq:ricci-soliton} yields $\mrm{div}(X) = \tr D$.  We also note that in the homogeneous case, we only need pointwise estimates.  Therefore, we don't need the (integrated) $L^2$ inner product, only a single inner product, so it suffices to consider $Q(h) = \langle \Ro h + \Ric \circ h, h \rangle$.

\begin{cor}\label{cor:alg-stab-cond}
Let $(\mcl{S},g,\lambda,D)$ be an algebraic Ricci soliton.  The following condition implies that $g$ is strictly linearly stable:
\[ Q(h) < \half \tr D |h|^2. \]
\end{cor}

\begin{rem}
In estimating $\opA$, one may not want to use the Codazzi Laplacian, in which case we use \eqref{eq:koiso-gen} to get
\begin{equation}\label{eq:A-exp-2}
\begin{aligned}
(\opA h,h) 
&= -(\nabla^*\nabla h,h)  - (\mfr{Ric}(h),h) + 2 \lambda \|h\|^2 \\
&\quad   - \half \mrm{div}(X) \|h\|^2 + 2 (\Ric \circ h,h ) - 2\lambda \|h\|^2 \\
&= -\|\nabla h\|^2 + 2(\Ro h,h) - \half \mrm{div}(X) \|h\|^2.
\end{aligned}
\end{equation}
This gives another condition for stability, $(\Ro h,h) < \frac{1}{4} \mrm{div}(X) \|h\|^2$.  It has the advantage of only requiring estimates of the Riemann curvature, but the estimates must be better than those for $Q$.
\end{rem}

We also describe conditions for stability when the sectional curvature is bounded.

\begin{prop}%\label{prop:sec-stab}
Let $(\mcl{M},g,\lambda,X)$ be a Ricci soliton with constant scalar curvature, and suppose that $\sec \leq K \leq 0$ for some constant $K$.  A sufficient condition for $g$ to be strictly linearly stable is that
\[ (n-2) K < \half \mrm{div}(X), \]
or $(n-2) K < \half \tr D$ in the case of an algebraic soliton.
\end{prop}

\begin{proof}
Select an orthonormal basis that diagonalizes $h$ and consider the (unintegrated) operator $Q(h)$.
\begin{align*}
\langle \Ro h + \Ric \circ h,h \rangle 
&= \sum_{ijk\ell} R_{ijk\ell} h^{i\ell} h^{jk} + \sum_{ijk} R_i^j h_j^k h_k^i \\
&= \sum_{ij} R_{ijji} h_{ii} h_{jj} + \sum_i R_{ii} h_{ii}^2 \\
&= \sum_{i \ne j} \sec(i,j) h_{ii} h_{jj} + \sum_{i \ne j} \sec(i,j) h_{ii}^2 \\
&= \half \sum_{i \ne j} \sec(i,j) 2 h_{ii} h_{jj} + \half \sum_{i \ne j} \sec(i,j) h_{ii}^2 + \half \sum_{i \ne j} \sec(i,j) h_{jj}^2 \\ 
&= \half \sum_{i \ne j} \sec(i,j) (h_{ii} + h_{jj})^2 \\
&\leq \half K \sum_{i \ne j} (h_{ii} + h_{jj})^2 \\ 
&= \half K \sum_{ij} (h_{ii} + h_{jj})^2 - \half K \sum_{i} 4 h_{ii}^2 \\
&= \half K \sum_{ij} (h_{ii}^2 + 2 h_{ii} h_{jj} + h_{jj}^2) - 2K \sum_{i} h_{ii}^2 \\
&= (n-2) K |h|^2 + K |\tr h|^2
\end{align*}
Now, we use \eqref{eq:A-exp-1} above.
\begin{align*}
(\opA h,h) 
&= -\half\|T\|^2 - \|\delta h\|^2 + (\Ro h,h) + (\Ric \circ h,h) - \half \mrm{div}(X) \|h\|^2 \\
&= -\half\|T\|^2 - \|\delta h\|^2 + (n-2) K \|h\|^2 + K \|\tr h\|^2 - \half \mrm{div}(X) \|h\|^2 \\
&\leq (n-2) K \|h\|^2 - \half \mrm{div}(X) \|h\|^2
\end{align*}
We want this to be less than zero, so the result follows.  
\end{proof}

\begin{rem}
If we take $K=0$, then the condition reduces to $\mrm{div}(X) > 0$ or $\tr D > 0$.  Also, we can allow for $K>0$ by using the inequality $|\tr h|^2 \leq n |h|^2$.  The condition then becomes 
\[ (n-1) K < \smfrac{1}{4} \mrm{div}(X), \]
but it is not clear how useful this is.
\end{rem}

\section{Nilsolitons and solvsoliton extensions}\label{sec:extensions}

Simply connected solvable Lie groups provide a rich class of non-compact spaces that often admit Ricci solitons and Einstein metrics.  Indeed, as mentioned above, all known examples of non-compact homogeneous expanding Ricci solitons can be realized as left-invariant metrics on solvable Lie groups.  Building on the work of Heber \cite{Heber1998}, Lauret has exploited the algebraic structure of these spaces to prove many powerful results regarding these metrics.  Before we state these results, recall that, given a Lie algebra $\mfr{n}$ with $\mfr{a} \subset \Der(\mfr{n})$, the semi-direct product $\mfr{n} \rtimes \mfr{a}$ has the bracket structure
\begin{equation}\label{eq:semi-direct}
[X,Y] = 0, \qquad 
[A,B] = 0, \qquad 
[A,X] = A(X), 
\end{equation}
for all $A,B \in \mfr{a}$, $X,Y \in \mfr{n}$.  Also recall that the \textit{mean curvature vector} of a metric Lie algebra $\mfr{n} \rtimes \mfr{a}$ is $H \in \mfr{a}$ such that $\langle H,A \rangle = \tr \ad_A$ for all $A \in \mfr{a}$.  

\begin{thm}[The structure and uniqueness of solvsolitons \cite{Lauret2011-sol}]\label{thm:lauret} 
\begin{itemize} 
\item[] \hphantom{-}
\item[(a)] Let $(\mfr{n}, \langle \cdot,\cdot \rangle_\mfr{n})$ be a nilsoliton with Ricci operator $\Ric_\mfr{n} = \lambda \, \id_\mfr{n} + D_\mfr{n}$, $\lambda < 0$ and $D_\mfr{n} \in \Der(\mfr{n})$.  Consider an abelian subalgebra $\mfr{a}$ of $\langle \cdot,\cdot \rangle_\mfr{n}$-symmetric derivations of $\mfr{n}$.  Then the simply connected solvmanifold $\mcl{S}$ with Lie algebra $\mfr{s} := \mfr{n} \rtimes \mfr{a}$ has an inner product $\langle \cdot,\cdot \rangle$ that is a solvsoliton with $\Ric = \lambda \, \id + D$.  This inner product satisfies
\[ \langle \cdot,\cdot \rangle|_{\mfr{n} \times \mfr{n}} = \langle \cdot,\cdot \rangle_\mfr{n}, \quad
 \langle A,B \rangle = -\smfrac{1}{\lambda} \tr (AB) \text{ for all } A,B \in \mfr{a}, \quad 
 \mfr{n} \perp \mfr{a}, \] 
and the derivation $D \in \Der(\mfr{s})$ satisfies
\[ D|_{\mfr{n}} = D_\mfr{n} - \ad_H|_{\mfr{n}}, \quad 
   D|_{\mfr{a}} = 0, \]
where $H$ is the mean curvature vector of $S$.  Furthermore, the inner product is Einstein if and only if $D_\mfr{n} \in \mfr{a}$, or equivalently, $D_\mfr{n} = \ad_H|_\mfr{n}$.

\item[(b)] All solvsolitons are of the form described in (a), with $\mfr{n}=[\mfr{s},\mfr{s}]$.

\item[(c)] Let $\mcl{S}$ and $\mcl{S}'$ be two solvsolitons which are isomorphic as Lie groups.  Then $\mcl{S}$ is isometric to $\mcl{S}'$ up to scaling.
\end{itemize}
\end{thm}

This essentially says that nilsolitons are the nilpotent parts of solvsolitons.  In particular, given a nilsoliton metric on a nilpotent Lie algebra $\mfr{n}$, it is always possible to construct a one-dimensional extension of $\mfr{n}$ that admits an Einstein metric.  We can form the solvable Lie algebra $\mfr{s} := \mfr{n} \rtimes \R D$, where $D$ is the nilsoliton derivation, and its Einstein metric is the one described in Part (a) of the Theorem.

Our goal in this section is to relate the linear stability of a left-invariant $\lambda$-Einstein metric on a solvable Lie group $\mcl{S}$ with that of the corresponding nilsoliton in the case of a codimension-one nilradical $\mcl{N}$.  Namely, we suppose that the Einstein metric satisfies the condition of Proposition \ref{prop:stab-conds}: $\langle \Ro h, h\rangle < -\lambda |h|^2$.  Then we then use this to estimate the spectrum of the operator $Q$ on the nilradical.  This comes down to comparing the eigenvalues of the soliton derivation to its trace.  Note again that all estimates are pointwise as the fixed-point metrics are homogeneous.  Also note that the Einstein metric will be a fixed point of the flow
\[ \partial_t g = -2 \Rc + 2 \lambda g \]
for metrics on $\mcl{S}$, and the nilsoliton is a fixed point of the flow for metrics on $\mcl{N}$:
\[ \partial_t g = -2 \Rc + 2 \lambda g + \mcl{L}_{X_\mfr{n}} g. \]

In terms of notation, geometric quantities on $\mfr{s}$ will not have subscripts, but those on $\mfr{n}$ or $\mfr{a}$ will.

\subsection{Curvature}\label{subsec:solv-curv}

Suppose that $\dim \mfr{n} = n$ and $\dim \mfr{a} = m$, and let 
\[ \{X_1,\dots,X_n,A_1,\dots,A_m\} \]
be an orthonormal basis of $\mfr{s} = \mfr{n} \rtimes \mfr{a}$ and let $U,V \in \mfr{s}$, $A,B,C,D \in \mfr{a}$, and $W,X,Y,Z \in \mfr{n}$.  In this section, we will use Roman indices for elements of $\mfr{n}$, Greek indices for elements of $\mfr{a}$.  

Recall that the covariant derivative corresponding to a left-invariant metric is given by
\[ 2 \nabla_U V = \ad_U V - \ad_U^* V - \ad_V^* U = \ad_U V + 2 \sigma(U,V), \]
where $\sigma(U,V) := -\half(\ad_U^* V + \ad_V^* U)$, see \cite{Besse2008,CheegerEbin2008}.  Write $\mrm{pr}_\mfr{a} \sigma(X,Y)$ for the orthogonal projection of $\sigma(X,Y)$ onto $\mfr{a}$.  Using that $\ad_A = \ad_A^*$, $\ad_A^*|_\mfr{a} = 0$, and $\ad_X^*|_\mfr{a}=0$, it is easy to see that
\begin{equation}\label{eq:nabla-s}
\begin{aligned}
\nabla_A B &= 0 \\
\nabla_A X &= 0 \\
\nabla_X A &= - \ad_A X \\
\nabla_X Y &= \nabla_X^\mfr{n} Y + \mrm{pr}_\mfr{a} \sigma(X,Y).
\end{aligned}
\end{equation}
The Riemann curvature tensor satisfies
\[ R(T,U,V,W) 
= \langle \nabla_T V,\nabla_U W \rangle 
- \langle \nabla_U V,\nabla_T W \rangle
- \langle \nabla_{[T,U]} V, W \rangle. \]
Using this and \eqref{eq:nabla-s},
\begin{equation}\label{eq:riem-s}
\begin{aligned}
R(A,B,C,D) &= 0 \\ 
R(A,B,C,X) &= 0 \\
R(A,B,X,Y) &= 0 \\
R(A,X,B,Y) &= \langle \ad_B \ad_A X,Y \rangle, \\
R(X,Y,Z,A) &= - \langle \nabla_X^\mfr{n} Z,\ad_A Y \rangle 
              + \langle \nabla_Y^\mfr{n} Z,\ad_A X \rangle 
              - \langle \mrm{pr}_\mfr{a} \sigma([X,Y],Z),A \rangle \\
R(X,Y,Z,W) & = R_\mfr{n}(X,Y,Z,W) \\
&\quad + \langle \mrm{pr}_\mfr{a} \sigma(X,Z),\mrm{pr}_\mfr{a} \sigma(Y,W) \rangle 
       - \langle \mrm{pr}_\mfr{a} \sigma(Y,Z),\mrm{pr}_\mfr{a} \sigma(X,W) \rangle.
\end{aligned}
\end{equation}
The Ricci curvature (see also \cite[Lemma 1.4]{Wolter1991} or \cite[Lemma 4.4]{Heber1998}) is
\begin{equation}\label{eq:ric-s}
\begin{aligned}
\Rc(A,B) &= -\tr (\ad_A \circ \ad_B) \\
\Rc(X,A) &= 0 \\
\Rc(X,Y) &= \Ric_\mfr{n}(X,Y) - \langle \ad_H X,Y \rangle.
\end{aligned}
\end{equation}

Since we're ultimately interested in the operators $\Ro$ and $Q$, we calculate $\langle \Ro h,h \rangle$ and $\langle \Ric \circ h,h \rangle$.  The idea is to exploit the product structure of the algebra.  First, note that a symmetric 2-tensor has the form
\begin{equation}\label{eq:h-decomp}
h = \left(
\begin{BMAT}{ccc.ccc}{ccc.ccc}
 &            & & &                 & \\
 & h_{ij}     & & & h_{i\beta}      & \\
 &            & & &                 & \\
 &            & & &                 & \\
 & h_{i\beta} & & & h_{\alpha\beta} & \\
 &            & & &                 &
\end{BMAT} \right), 
\end{equation}
and
\begin{equation}\label{eq:h-norm-n}
|h|^2
= \sum_{ij} h_{ij}^2 
+ 2 \sum_{i\beta} h_{i\beta}^2
+ \sum_{\alpha\beta} h_{\alpha\beta}^2.
\end{equation}  
We use capital Roman letters to refer to indices ranging over both $1,\dots,n$ (from $\mfr{n}$) and $1,\dots,m$ (from $\mfr{a}$), that is, $\sum_I = \sum_i + \sum_\alpha$.
Then summing over repeated indices,
\begin{equation}\label{eq:Ro-expand}
\begin{aligned}
\langle \Ro h,h \rangle
&= R_{IJKL} h_{IL} h_{JK} \\
&= R_{ijk\ell} h_{i\ell} h_{jk} 
   + R_{ijk\delta} h_{i\delta} h_{jk} 
   + R_{ij\gamma\ell} h_{i\ell} h_{j\gamma}
   + \cancel{R_{ij\gamma\delta}} h_{i\delta} h_{j\gamma} \\
&\quad + R_{i\beta k\ell} h_{i\ell} h_{\beta k} 
       + R_{i\beta k\delta} h_{i\delta} h_{\beta k} 
       + R_{i\beta\gamma\ell} h_{i\ell} h_{\beta\gamma}
       + \cancel{R_{\alpha j \gamma \delta}} h_{i\delta} h_{\beta\gamma} \\
&\quad + R_{\alpha jk\ell} h_{\alpha\ell} h_{jk} 
       + R_{\alpha jk\delta} h_{\alpha\delta} h_{jk} 
       + R_{\alpha j\gamma \ell} h_{\alpha\ell} h_{j \gamma} 
       + \cancel{R_{\alpha j \gamma \delta}} h_{\alpha \delta} h_{j \gamma} \\
&\quad + \cancel{R_{\alpha\beta k\ell}} h_{\alpha\ell} h_{\beta k} 
       + \cancel{R_{\alpha\beta k\delta}} h_{\alpha\delta} h_{\beta k} 
       + \cancel{R_{\alpha\beta\gamma\ell}} h_{\alpha\ell} h_{\beta\gamma} 
       + \cancel{R_{\alpha\beta\gamma\delta}} h_{\alpha\delta} h_{\beta\gamma}
\end{aligned}
\end{equation}
and
\begin{equation}\label{eq:Rico-expand}
\begin{aligned}
\langle \Ric \circ h,h \rangle
&= R_{IJ} h_{JK} h_{KI} \\
&= R_{ij} h_{jk} h_{ki}
 + R_{ij} h_{j\gamma} h_{\gamma i}
 + \cancel{R_{i\beta}} h_{\beta k} h_{ki}
 + \cancel{R_{i\beta}} h_{\beta\gamma} h_{\gamma i} \\
&\quad + \cancel{R_{\alpha j}} h_{jk} h_{k\alpha}
 + \cancel{R_{\alpha j}} h_{j\gamma} h_{\gamma \alpha}
 + R_{\alpha\beta} h_{\beta k} h_{k\alpha}
 + R_{\alpha\beta} h_{\beta\gamma} h_{\gamma\alpha}
\end{aligned}
\end{equation}  
These types of expansions will also appear later in the paper.

\subsection{Codimension-one nilradicals}

Let $(\mfr{s},\langle \cdot,\cdot \rangle)$ be $\lambda$-Einstein and suppose that $\langle \Ro \tilde{h},\tilde{h} \rangle < -\lambda |\tilde{h}|^2$ for all symmetric 2-tensors $\tilde{h}$ on $\mfr{s}$.  By Corollary \ref{cor:alg-stab-cond}, the metric is strictly linearly stable.  Suppose further that $\mfr{s}$ has a codimension-one nilradical $\mfr{n}$, which admits a nilsoliton by Theorem \ref{thm:lauret}.  We want to show that $Q_\mfr{n}(h) < \half \tr D |h|_\mfr{n}^2$, where $h$ now lives on $\mfr{n}$.  Such an $h$ can be thought of as a certain type of tensor on $\mfr{s}$, where
\[ h(X,A) = h(A,X) = h(A,B) = 0 \]
for all $X \in \mfr{n}$, $A,B \in \mfr{a}$.  This will allow us to use the estimate for $\langle \Ro h,h \rangle$.

We begin with an observation about the quantity $\mrm{pr}_\mfr{a} \sigma(X,Y)$.

\begin{lem}\label{lem:1-dim-D}
Let $(\mfr{n},\langle \cdot,\cdot \rangle_\mfr{n})$ be a nilsoliton with $\Ric = \lambda \, \id + D$.  Consider the 1-dimensional Einstein extension $\mfr{s} = \mfr{n} \rtimes \R A$, where $|A|^2 = 1$.  Then we have 
\[ \ad_A = \frac{1}{\sqrt{\tr D}} D, \qquad 
   \mrm{pr}_\mfr{a} \sigma(X,Y) = \frac{1}{\sqrt{\tr D}} \langle DX, Y \rangle A. \]
\end{lem}

\begin{proof}
Let $A$ be a unit vector spanning $\mfr{a}$, so that $\ad_A = c D$ for some $c \in \R$ to be determined, and let $X,Y \in \mfr{n}$.  The decomposition of the covariant derivative from \eqref{eq:nabla-s} involves the quantity $\mrm{pr}_\mfr{a} \sigma(X,Y)$, which in this case is essentially the second fundamental form, $II(X,Y)$, thought of as a real-valued bilinear map.  We compute
\begin{align*}
\mrm{pr}_\mfr{a} \sigma(X,Y)
&= II(X,Y) A \\
&= \langle \nabla_X Y,A \rangle A \\
&= \langle \half \ad_X Y - \half \ad_X^* Y - \half \ad_Y^* X, A \rangle A \\
&= \half \langle Y, \ad_A X \rangle A + \half \langle X, \ad_A Y \rangle A \\
&= \langle cDX, Y \rangle A
\end{align*}
This means that on the endomorphism level, $II = cD = \ad_A$.  To find $c$, we use the form of the metric on $\mfr{a}$ as in Theorem \ref{thm:lauret} (a):
\[ 1 = |A|^2 = -\frac{1}{\lambda} \tr \ad_A^2 = -\frac{c^2}{\lambda} \tr D^2, \]
and since $-\lambda = \tr D^2/\tr D$, we have $c = \sqrt{1/\tr D}$.  The desired formula follows.
\end{proof}

Now, using \eqref{eq:Ro-expand} together with \eqref{eq:riem-s} and the Lemma, it is not hard to compute
\begin{equation}\label{eq:ro-relations}
\begin{aligned}
\langle \Ro h,h \rangle  
%&= \langle \Ro h,h \rangle_\mfr{n} \\
%&\quad + \sum_{ijk\ell} g_\mfr{s}(\sigma(X_i,X_\ell)|_\mfr{a},\sigma(X_k,X_j)|_\mfr{a}) h(X_k,X_\ell) h(X_i,X_j) \\ 
%&\quad - \sum_{ijk\ell} g_\mfr{s}(\sigma(X_i,X_j)|_\mfr{a},\sigma(X_k,X_\ell)|_\mfr{a}) h(X_k,X_\ell) h(X_i,X_j) \\
%&= \langle \Ro h,h \rangle_\mfr{n}
% + \frac{1}{\tr D} \sum_{i,j,k,\ell} D_\ell^i D_j^k h_{k\ell} h_{ij} 
% - \frac{1}{\tr D} \sum_{i,j,k,\ell} D_j^i D_\ell^k h_{k\ell} h_{ij} \\
&= \langle \Ro_\mfr{n} h,h \rangle_\mfr{n} + \frac{1}{\tr D} ( \tr(D \circ h)^2 - \langle D,h \rangle^2 ). 
\end{aligned}
\end{equation}
Similarly, from \eqref{eq:Rico-expand}, \eqref{eq:ric-s}, and the Lemma we obtain
\begin{equation}\label{eq:rico-relations}
\begin{aligned}
\langle \Ric \circ h,h \rangle
%&= \sum_{ijk} [ \Rc^\mfr{n}(X_i,X_j) - g_\mfr{n}(\ad_H X_i,X_j) ] h(X_j,X_k) h(X_k,X_i) \\
&= \langle \Ric_\mfr{n} \circ h,h \rangle_\mfr{n}^2 + \langle D \circ h,h \rangle.
\end{aligned}
\end{equation}

Now we can prove the main result of this section.

\begin{prop}%\label{prop:ext-stab}
Let $\mfr{s}$ be a solvable Lie algebra with codimension-one nilradical $\mfr{n} = [\mfr{s},\mfr{s}]$.  If a $\lambda$-Einstein metric on $\mfr{s}$ satisfies $\langle \Ro h,h \rangle < -\lambda |h|^2$ for all symmetric 2-tensors $h$, then a sufficient condition for the nilsoliton on $\mfr{n}$ to be strictly linearly stable is that
\[ \max\{d_i\} < \frac{1}{2+\sqrt{2}} \tr D, \]
where $d_1,\dots,d_n$ are the eigenvalues of the nilsoliton derivation $D$.
\end{prop}

\begin{proof}
Our assumption on $\langle \cdot,\cdot \rangle$ says that $Q(h) < 0$, and we want to show that $Q_\mfr{n}(h) < \half \tr D |h|^2$.  Using \eqref{eq:ro-relations} and \eqref{eq:rico-relations} it is easy to see that
\begin{align*}
Q_\mfr{n}(h)
&= \langle \Ro_\mfr{n} h,h \rangle_\mfr{n} + \langle \Ric_\mfr{n} \circ h,h \rangle_\mfr{n} \\
&= Q(h) - \frac{1}{\tr D} ( \tr(D \circ h)^2 - \langle D,h \rangle^2 ) - \langle D \circ h,h \rangle \\
&< \frac{1}{\tr D} (\langle D,h \rangle^2 - \tr(D \circ h)^2 + \tr D \langle D \circ h,h \rangle).
\end{align*}
Therefore, we need to show that
\[ \langle D,h \rangle^2 - \tr(D \circ h)^2 + \tr D \langle D \circ h,h \rangle
\overset{?}{<} \smfrac{1}{2} (\tr D)^2 |h|^2_\mfr{n}. \]

Select an orthonormal basis of $\mfr{n}$ such that $D$ is diagonalized, with (positive) diagonal entries $\{d_1,\dots,d_n\}$.  Then we have
\begin{align*}
&\langle D,h \rangle^2 - \tr(D \circ h)^2 + \tr D \langle D \circ h,h \rangle \\
&= \sum_{ijk\ell} D_j^i D_\ell^k h_{k\ell} h_{ij}
  -\sum_{ijk\ell} D_\ell^i D_j^k h_{k\ell} h_{ij} 
  + \tr D \sum_{ijk} D_j^i h_{jk} h_{ki} \\
%&= \sum_{ijk\ell} d_i \delta_j^i d_k \delta_\ell^k h_{k\ell} h_{ij}
%  -\sum_{ijk\ell} d_i \delta_\ell^i d_k \delta_j^k h_{k\ell} h_{ij} 
%  + \tr D \sum_{ijk} d_i \delta_j^i h_{jk} h_{ki} \\
&= \sum_{ij} d_i d_j h_{ii} h_{jj}
  -\sum_{ij} d_i d_j h_{ij}^2 
  + \tr D \sum_{i,j} d_i h_{ij}^2 \\
%&\leq \sum_{ij} \left[ \half d_i d_j h_{ii}^2 + \half d_i d_j h_{jj} \right]
%     -\sum_{ij} d_i d_j h_{ij}^2 
%     + \tr D \sum_{ij} d_i h_{ij}^2 \\
&\leq \sum_{ij} d_i d_j h_{ii}^2 
  -\sum_{ij} d_i d_j h_{ij}^2 
  + \tr D \sum_{ij} d_i h_{ij}^2 \\
&= \tr D \sum_{i} d_i h_{ii}^2 
  -\sum_{i} d_i^2 h_{ii}^2 
  -\sum_{i \ne j} d_i d_j h_{ij}^2 
  + \tr D \sum_{i} d_i h_{ii}^2 
  + \tr D \sum_{i\ne j} d_i h_{ii}^2 \\
&= \sum_{i} (2 d_i \tr D  - d_i^2) h_{ii}^2 
  +\sum_{i \ne j} (d_i \tr D - d_i d_j) h_{ij}^2 \\  
&\overset{?}{<} 
   \sum_{i} \half (\tr D)^2  h_{ii}^2 
  +\sum_{i \ne j} \half (\tr D)^2 h_{ij}^2 \\
&= \smfrac{1}{2} (\tr D)^2 |h|^2.    
\end{align*}
The first inequality uses the arithmetic/geometric mean inequality.  For the second inequality to hold, it is enough for the following two inequalities to hold:
\begin{align}
2 d_i \tr D  - d_i^2 &\overset{?}{<} \smfrac{1}{2} (\tr D)^2, \label{eq:ineq1} \\ 
d_i \tr D - d_i d_j  &\overset{?}{<} \smfrac{1}{2} (\tr D)^2. \label{eq:ineq2}
\end{align}
We rewrite \eqref{eq:ineq1} as follows:
\begin{align*}
0 
&\overset{?}{<} \smfrac{1}{4} (\tr D)^2 - d_i \tr D + \half d_i^2 \\
&= ( \half \tr D - d_i)^2 - \half d_i^2 \\
&= ( \half \tr D - (1+\smfrac{1}{\sqrt{2}}) d_i)( \half \tr D + (-1+\smfrac{1}{\sqrt{2}}) d_i) 
\end{align*}
This is equivalent to one the following statements
\[ \tr D - (2+\sqrt{2}) d_i < 0 \quad \text{and} \quad \tr D + (-2+\sqrt{2}) d_i < 0 \]
or
\[ \tr D - (2+\sqrt{2}) d_i > 0 \quad \text{and} \quad \tr D + (-2+\sqrt{2}) d_i > 0 \]
The second condition in the first statement is never satisfied, since $\tr D > d_i$ for all $i$, so we consider the second statement.  The inequalities there can be rewritten
\[ d_i < \frac{1}{2+\sqrt{2}} \tr D, \qquad d_i < \frac{1}{2-\sqrt{2}} \tr D, \]
and the second is always satisfied.  Therefore \eqref{eq:ineq1} is equivalent to $d_i < \frac{1}{2+\sqrt{2}} \tr D$.

Now, notice that \eqref{eq:ineq2} follows from \eqref{eq:ineq1}:
\[ d_i \tr D - d_i d_j < \frac{1}{2+\sqrt{2}} (\tr D)^2 < \frac{1}{2} (\tr D)^2. \qedhere \]
\end{proof}

\subsection{Open sets of stable solvsolitons}

Here we verify the last statement in Theorem \ref{thm:stab-all} by applying Corollary \ref{cor:alg-stab-cond} and observing that $Q$ and $D$ vary continuously as we vary the Lie structure of the underlying solvable Lie group.  In all of the following, we normalize our soliton metric by fixing the soliton constant $\lambda$ appearing in the equation $\Ric = \lambda \id + D$.

Consider a nilpotent Lie algebra $\mfr n$ which admits a nilsoliton metric.  As described in Theorem \ref{thm:lauret}, for any abelian subalgebra $\mfr a$ of symmetric derivations, the solvable Lie algebra $\mfr s = \mathfrak n \rtimes \mfr a$ admits a solvsoliton.  Further, the solvsoliton derivation $D$ on $\mfr{s}$ satisfies
\begin{equation}\label{eq:der-s}
D|_\mfr{n} = D_\mfr{n} - \ad_H, 
\end{equation}
where $H$ is the mean curvature of $\mfr s$ and $D_\mfr{n}$ is the soliton derivation of $\mfr n$.  The metric on $\mfr s$ is Einstein precisely when $\ad_H = D_\mfr{n}$. 

To see more precisely how a choice of abelian subalgebra of symmetric derivations effects the resulting solvsoliton, we recall a Corollary of Lauret's Theorem; see, e.g., \cite{Will2011}.  Let $\sym(\mfr{n})$ denote symmetric linear maps on $\mfr{n}$, with respect to the nilsoliton metric.  Define the \textit{rank} of the nilpotent Lie algebra to be $\rank(\mfr{n}) = \dim(\mfr{a}_\mfr{n})$, where $\mfr{a}_\mfr{n} \subset \Der(\mfr{n}) \cap \sym(\mfr{n})$ is a maximal abelian subalgebra. 

\begin{prop}[\cite{Will2011}]\label{prop:sol-space}
Let $0 < k < \rank(\mfr{n})$.  The moduli space of $(n+k)$-dimensional solvsoliton extensions of $\mfr{n}$, up to isometry and scaling, is parametrized by
\[ \Gr_k(\mfr{a}_\mfr{n})/W_\mfr{n}, \]
where $W_\mfr{n}$ is the Weyl group of $G_\mfr{n} := \{g \in \Aut(\mfr{n}) \mid g^t \in \Aut(\mfr{n})\}$.
\end{prop}

Note that when $k=0$ or $k=\rank(\mfr{a}_\mfr{n})$ the moduli space is a point consisting of the original nilsoliton in the former case, or a single Einstein metric in the latter.

Now we consider changing $\mfr{a} \subset \mfr{a}_\mfr{n}$, which amounts to moving within the Grassmannian $\Gr_k(\mfr{a}_\mfr{n})$.  The mean curvature vector $H \in \mfr a$, which satisfies $\langle H,A \rangle = \tr \ad_A$ for all $A\in \mfr s$, can be alternately described as the unique vector $H \in \mfr a$ that solves
\[ -\frac{1}{\lambda} \tr (\ad_H \circ \ad_A) = \tr \ad_A \quad \mbox{ for all } A \in \mfr a. \]
From this and \eqref{eq:der-s} we see that the solvsoliton derivation $D$ of $\mfr s$ depends continuously on $\mfr{a}$.

We also observe that the geometric quantities $\Ro$ and $\Ric$ depend continuously on $\mfr a$, since these quantities are completely determined by the structure constants of the Lie algebra, which are determined by $\mfr{a}$ and $\mfr{n}$; see Subsection \ref{subsec:solv-curv}.  Hence, $Q$ also depends continuously on $\mfr{a}$.

Finally, suppose that the soliton on $\mfr{s} = \mfr{n} \rtimes \mfr{a}$ satisfies the condition of Corollary \ref{cor:alg-stab-cond}.  As long as $0 < \dim \mfr{a} < \rank(\mfr{n})$, the condition of Corollary \ref{cor:alg-stab-cond} is preserved under sufficiently small perturbations of $\mfr{a}$.  The result is an open set in $\Gr_k(\mfr{a}_\mfr{n})$ of stable metrics.  Taking the quotient by the finite group $W_\mfr{n}$, we have an open set of stable metrics in the moduli space of solvsoliton extensions of $\mfr{n}$.

\section{Two-step solvsolitons}\label{sec:abelian}

Consider an abelian nilsoliton $\mfr{n}$, which we identify with $\R^n$ through an orthonormal basis $\{X_i\}$.  Since the metric is flat, $\Ric_\mfr{n} = 0$ and we may choose any $\lambda < 0$ as the soliton constant.  If we pick $\lambda = -1$, then the soliton derivation is $D_\mfr{n} = \id_\mfr{n}$.  Strict linear stability of these solitons follows from Corollary \ref{cor:alg-stab-cond}, since $Q(h) = 0 < |h|^2$ when $h \ne 0$.

Solvable extensions of abelian nilpotent Lie groups are called two-step solvable Lie groups.  For example, real hyperbolic space is obtained by adjoining to $\R^n$ the 1-dimensional algebra spanned by the identity map (which is a derivation).  Similarly, any other linear map yields a two-step solvable space admitting a solvsoliton.  Using \ref{prop:sec-stab}, we construct an open set of strictly linearly stable two-step solvsolitons.

\subsection{Construction}

As described in see \cite[Section 3.1]{Will2011}, the algebra of derivations of $\mfr{n}$ is $\Der(\mfr{n}) = \mfr{gl}(n,\R)$, and with respect to the chosen orthonormal basis, a maximal abelian subalgebra $\mfr{a}_0 \subset \mfr{gl}(n,\R) \cap \sym(\mfr{n})$ is the $n$-dimensional space of all diagonal matrices.  A choice of a subalgebra $\mfr{a} \subset \mfr{a}_0$ defines a solvsoliton extension of $\mfr{n}$ that is unique up to the action of the symmetric group $S_n$ on $\mfr{a}_0$ (which permutes diagonal entries).  Solvsoliton extensions of $\mfr{n}$ are therefore parametrized by the union of the spaces $\Gr_m \R^n/S_n$, where $0 \leq m \leq n$ is the dimension of $\mfr{a}$.  When $m=1$ this space is $\mbb{RP}^{n-1}/S_n$, and when $m=0$ or $m=n$ this space is a point consisting of only the nilsoliton or an Einstein metric, respectively.

Select a subalgebra $\mfr{a} \subset \mfr{a}_0$ of dimension $m$, which we describe with an orthonormal basis $\{A_\alpha\}$ consisting of diagonal matrices: 
\[ A_\alpha = \mrm{diag}(a_{\alpha 1},\dots,a_{\alpha n}) \qquad \alpha = 1,\dots,m. \]
(As before, Roman indices will refer to vectors in $\mfr{n}=\R^n$ and Greek to vectors in $\mfr{a}$).  By Theorem \ref{thm:lauret}, we get a solvsoliton metric on $\mfr{s} := \R^n \rtimes \mfr{a}$ that makes $\R^n$ and $\mfr{a}$ orthogonal, and on elements of $\mfr{a}$ it is given by
\begin{equation}\label{eq:abelian-onb}
\delta_{\alpha\beta} 
= \langle A_\alpha,A_\beta \rangle 
:= -\frac{1}{\lambda} \tr(A_\alpha A_\beta)
= \sum_{i=1}^n a_{\alpha i} a_{\beta i}.
\end{equation}
In particular, we have
\[ 1 = |A_\alpha|^2 = \sum_{i=1}^n (a_{\alpha i})^2. \]

The Lie algebra structure on $\R^n \rtimes \mfr{a}$ is the semi-direct product, so we have the bracket relations described in \eqref{eq:semi-direct}.  In terms of structure constants, we have
\[ c_{ij}^k = c_{ij}^\gamma = 0, \qquad 
c_{\alpha\beta}^k = c_{\alpha\beta}^\gamma = 0, \qquad
c_{\alpha i}^j = \delta_{ij} a_{\alpha i}, \qquad
c_{\alpha i}^\gamma = 0. \]
From this we can compute the transposes of the adjoint maps.  For example,  %Recall that in general, $(\ad_{E_i}^*)_j^k = c_{ik}^j$.  Then
\[ \ad_{X_i}^* X_j = - \delta_{ij} \sum_{\gamma=1}^m a_{\gamma i} A_\gamma, \qquad
   \ad_{X_i}^* A_\alpha = 0. \]

The mean curvature vector is $H \in \mfr{a}$ such that $\langle H,A \rangle = \tr \ad_A$ for all $A \in \mfr{a}$.  From this it is easy to compute that 
\[ H^\alpha 
= \langle H,A_\alpha \rangle \\
= \tr \ad_{A_\alpha} \\
= \tr A_\alpha \]
and so
\[ H 
= \sum_{\alpha} (\tr A_\alpha) A_\alpha.
\]
Its adjoint satisfies
\[ \ad_H|_\mfr{n} = \sum_\alpha (\tr A_\alpha) \ad_{A_\alpha}, \qquad
\ad_H|_\mfr{a} = 0. \]
This means $\ad_H$ is diagonal in any orthonormal basis, and the diagonal entries are $(\ad_H)_{ii} = \sum_\alpha (\tr A_\alpha) a_{\alpha i}$ or $0$.  We also have that
\[ |H|^2
= \tr \ad_H 
= \sum_\alpha (\tr A_\alpha)^2.
\]
We can say more about this quantity.

\begin{lem}\label{lem:abelian-norm-H}
The mean curvature vector corresponding to a two-step solvsoliton satisfies $|H|^2 \leq n$, with equality if and only if the metric is Einstein.
\end{lem}

\begin{proof}
Think of elements of the orthonormal basis $\{A_1,\cdots,A_m\}$ as column vectors and put them into a matrix $A$ (which satisfies $A^* A = I_m$, by \eqref{eq:abelian-onb}).  Next, complete this basis to an orthonormal basis of the $n$-dimensional maximal abelian subalgebra $\mfr{a}_0$.  Adding these new columns to $A$, we get a matrix $B \in \mrm{O}(n)$.  Then $B B^* = I_n$ and so
\[ \delta_{ij} = (B B^*)_{ij} = \sum_{\gamma} a_{\gamma i} a_{\gamma j}. \]
Summing over $i$ and $j$ we get
\[
n 
= \sum_{ij\gamma} a_{\gamma i} a_{\gamma j} 
= \sum_\gamma (\tr A_\gamma)^2 
= |H|^2 + \sum_{\gamma=m+1}^n (\tr A_\gamma)^2.
\]
This implies that $|H|^2 \leq n$.  

Noting that the nilsoliton derivation $D = \id_\mfr{n} \in \mfr{a}_0$ is represented by the vector
\[ D = \begin{pmatrix} 1\\ \vdots\\ 1 \end{pmatrix}, \]
the above calculation also gives us the following equivalent statements.
\begin{align*}
|H|^2 = n
&\Longleftrightarrow \tr A_\gamma = 0 \text{ for all } \gamma = m+1,\dots,n \\
&\Longleftrightarrow D \perp A_\gamma \text{ for all } \gamma = m+1,\dots,n \\
&\Longleftrightarrow D \in \mrm{span}\{A_1,\dots,A_m\} = \mfr{a} 
\end{align*}
By Theorem \ref{thm:lauret} (a), this happens if and only if the extension is Einstein.
\end{proof}

\subsection{Curvature}

We want to compute the curvatures of $\mfr{s}$, so recall the formulas from Section \ref{sec:extensions}.  Since $\mfr{n} = \R^n$ is abelian, we have that $\nabla_X^\mfr{n} Y = 0$ for all $X,Y \in \mfr{n}$.  Also recall that $\sigma(U,V) = -\half(\ad_U^* V + \ad_V^* U)$, and that $\mrm{pr}_\mfr{a} \sigma(X,Y)$ is the orthogonal projection of $\sigma(X,Y)$ onto $\mfr{a}$.  Using the above formulas, it is easy to see that
\[ \mrm{pr}_\mfr{a} \sigma(X_i,X_j)
= \sum_\alpha \langle \sigma(X_i,X_j), A_\alpha \rangle A_\alpha 
= \delta_{ij} \sum_{\alpha} a_{\alpha i} A_\alpha. \]
Now, using this and \eqref{eq:riem-s}, we calculate the Riemann curvatures.
\begin{align*}
R_{\alpha\beta\gamma\delta} 
&= 0 \\
R_{i\alpha j\beta}
&= \delta_{ij} a_{\alpha i} a_{\beta i} \\
R_{ijk\alpha}
&= 0 \\
R_{ijk\ell}
&= (\delta_{ik} \delta_{j\ell} - \delta_{i\ell} \delta_{jk}) \sum_{\alpha} a_{\alpha i} a_{\alpha j} 
\end{align*}

Now, let $U$ and $V$ be a pair of orthonormal vectors in $\mfr{s}$.  We will compute the sectional curvature $\sec(U,V)$ in a manner similar to \eqref{eq:Ro-expand}, with similar cancellations.  
First, we introduce the short-hand notation
\[ C_{i}\left(V\right) := \sum_{\alpha}a_{\alpha i}V^{\alpha}, \quad    
   C_{i}\left(U\right) := \sum_{\alpha}a_{\alpha i}U^{\alpha}. \]
Now, sum over all repeated indices.
\begin{align*}
\sec(U,V)
&= R(U,V,V,U) \\
&= R_{IJKL} U^I V^J V^K U^L \\
%&= R_{ijk\ell} U^i V^j V^k U^\ell 
% + R_{i\beta k\delta} U^i V^\beta V^k U^\delta 
% + R_{i\beta\gamma\ell} U^i V^\beta V^\gamma U^\ell \\
&\quad + R_{\alpha jk\delta} U^\alpha V^j V^k U^\delta 
 + R_{\alpha j\gamma \ell} U^\alpha V^j V^\gamma U^\ell \\
&= R_{ijk\ell} U^i V^j V^k U^\ell 
 + R_{i\alpha j\beta} U^i V^\alpha V^j U^\beta 
 + R_{i\alpha \beta j} U^i V^\alpha V^\beta U^j \\
&\quad + R_{\alpha ij \beta} U^\alpha V^i V^j U^\beta 
 + R_{\alpha i\beta j} U^\alpha V^i V^\beta U^j \\ 
&= (\delta_{ik} \delta_{j\ell} 
  -\delta_{i\ell} \delta_{jk}) a_{\alpha i} a_{\alpha j} U^i V^j V^k U^\ell 
 + \delta_{ij} a_{\alpha i} a_{\beta i} U^i V^\alpha V^j U^\beta \\ 
&\quad - \delta_{ij} a_{\alpha i} a_{\beta i} U^i V^\alpha V^\beta U^j
 - \delta_{ij} a_{\alpha i} a_{\beta i} U^\alpha V^i V^j U^\beta 
 + \delta_{ij} a_{\alpha i} a_{\beta i} U^\alpha V^i V^\beta U^j \\ 
&= a_{\alpha i} a_{\alpha j} \left( U^i V^j V^i U^j - \left( U^i V^j \right)^2 \right) \\
&\quad + a_{\alpha i} a_{\beta i} \left( 2 U^i V^i V^\alpha U^\beta 
- V^\alpha V^\beta \left( U^i \right)^2 
- U^\alpha U^\beta \left( V^i \right)^2 \right) \\
&= -\sum_\alpha \sum_{i<j} a_{\alpha i} a_{\alpha j} \left(U^i V^j - U^j V^i \right)^2 \\ 
&\quad + 2 \sum_i \left( C_i \left( V \right) U^i \right) \left( C_i \left( U \right) V^i \right) - \left(C_i \left( V \right) U^i \right)^2 - \left( C_i \left( U \right) V^i \right)^2 \\
&= - \sum_\alpha \sum_{i<j} a_{\alpha i} a_{\alpha j} \left( U^i V^j - U^j V^i \right)^2 
   - \sum_i  \left( \left( C_i \left( V \right) U^i \right) - \left(C_i \left( U \right) V^i \right) \right)^2
\end{align*}
We will use this to prove the main result of this section.

\begin{prop}\label{prop:abelian-sec}
There is an open set of strictly linearly stable solvsolitons whose nilradicals have codimension one, are abelian, and have dimension two or greater.
\end{prop}

\begin{proof}
When $m=1$, the above curvature formula reduces to 
\[ \sec(U,V)
= - \sum_{i<j} a_i a_j (U^i V^j - U^j V^i)^2 - \sum_i a_i^2 (U^i V^0 - U^0 V^i)^2, \]
where $0$ refers to the $\mfr{a}$-component of the vectors.  Clearly, if 
\begin{equation}\label{eq:abelian-cones}
a_i \leq 0 \text{ for all } i, \qquad \text{ or } a_i \geq 0 \text{ for all } i,
\end{equation}
then the sectional curvature is non-positive.  

Since $A = \mrm{diag}(a_1,\dots,a_n)$ satisfies $\sum_i a_i^2 =1$, we may think of $A$ as an element of the unit sphere $S^{n-1}$.  The set of $A$ satisfying \eqref{eq:abelian-cones} contains an open set on $S^{n-1}$.  As mentioned above, the moduli space of non-isometric solvsolitons is parametrized by $\mbb{RP}^{n-1}/S_n$.  So, after taking the quotient of $S^{n-1}$ by several finite groups, we still have an open set of non-positively curved solvsolitons in the moduli space.  From Theorem \ref{thm:lauret}, the soliton derivations satisfy $\tr D = n - |H|^2$, which is positive in the non-Einstein case by Lemma \ref{lem:abelian-norm-H}.  The single Einstein metric has negative curvature; see the next subsection.  Therefore, each metric in this set is stable by Proposition \ref{prop:sec-stab}.
\end{proof}

\subsection{Examples}\label{subsec:abelian-einstein}

Here we describe some properties of two-step solvable Einstein metrics.  Of particular interest are examples of metrics for which both Corollary \ref{cor:alg-stab-cond} and Proposition \ref{prop:sec-stab} fail to give stability.  Let $\dim \mfr{a} = m$, with $1 \leq m \leq n$.  The solvsoliton extension of $\R^n$ is determined by $\mfr{a}$, independent of any basis.  In the Einstein case, we can always pick a basis of $\mfr{a}$ such that $A_1 = D/\sqrt{n} = \id_\mfr{n}/\sqrt{n}$ (normalized as in Lemma \ref{lem:1-dim-D}), and $\tr A_\alpha = 0$ for $\alpha = 2,\dots,m$.  

\begin{ex}
When $m=1$, the extension is real hyperbolic space $H^{n+1}$ and the matrix we adjoin is $A_1 = \id_\mfr{n}/\sqrt{n}$.  Taking this a bit further, we see that every two-step solvable Einstein metric is an extension of an abelian algebra by a hyperbolic space.  Indeed, 
\[ \mfr{h}^{n+1} 
:= \mrm{Lie}(H^{n+1}) 
= \mrm{span}\{X_1,\dots,X_n,A_1\} \]
is an ideal in $\R^n \rtimes \mfr{a}$, so we have a short exact sequence:
\[ 0 
\longrightarrow 
\mfr{h}^{n+1} 
\longrightarrow 
\R^n \rtimes \mfr{a} 
\longrightarrow 
\mrm{span}\{A_2,\dots,A_m\}
\longrightarrow 
0. \]

When $n>1$, the hyperbolic space $H^{n+1}$ has constant negative curvature, so it is stable by Proposition \ref{prop:sec-stab}.  When $n=1$, one can check that $H^2$ is only weakly linearly stable.  Compact quotients of this space are dynamically stable, however, and proving this requires more detailed analysis of the null-space of the operator in \ref{eq:cnrf-lin}; see \cite[Lemma 5]{Knopf2009}.

Also, by \cite[Theorem 3]{Heintze1974}, a metric with strictly negative curvature must have $\dim \mfr{a} = 1$, so hyperbolic spaces are the only two-step solvable Einstein metrics with strictly negative curvature.
\end{ex}

\begin{ex}\label{ex:hyp-prod}
When $m=n$, it is convenient to represent the basis of $\mfr{a}$ as a matrix as in the proof of Lemma \ref{lem:abelian-norm-H}.  In this case we take $A := I_n$.  It is not hard to see that the resulting $2n$-dimensional space is isometric to the $n$-fold Riemannian product
\[ H^2 \times \cdots \times H^2. \]
This space is Einstein with $\lambda = -1$.  One can compute the eigenvalues of $\Ro$ and see that the largest is in fact $1$, so we only have weak linear stability, as with $H^2$.

Note that by \cite[Proposition 2.1]{Wolter1991}, any two-step solvable Einstein metric that has non-positive sectional curvature will be a product of $m$ real hyperbolic spaces.
\end{ex}

\begin{ex}
Let us consider two more concrete metrics.  Let $n=4$ and $m=3$, and consider the following matrix $A$ describing the basis for $\mfr{a}$ (again as in the proof of Lemma \ref{lem:abelian-norm-H}).
\[ A := \frac{1}{\sqrt{12}}
\begin{pmatrix}
 \sqrt{3} & 2 \sqrt{2} & 0 \\
 \sqrt{3} & -\sqrt{2} & \sqrt{6} \\
 \sqrt{3} & -\sqrt{2} & -\sqrt{6} \\
 \sqrt{3} & 0 & 0
\end{pmatrix} \]
The resulting Einstein metric does not satisfy the conditions of either Corollary \ref{cor:alg-stab-cond} or Proposition \ref{prop:sec-stab}.  Indeed, using \eqref{eq:Ro-expand} one can see that
\[ h := \mrm{diag}(-\smfrac{9}{8},-1,-1,1,0,1,1)
\quad \Longrightarrow \quad
\frac{\langle \Ro h,h \rangle}{\langle h,h \rangle} = \frac{1204}{1203} > 1 = -\lambda. \]
One can also see easily that the sectional curvatures have mixed sign.

We note that taking $A$ to be only the first two columns yields a soliton that is stable according to Corollary \ref{cor:alg-stab-cond}.  Thinking of the soliton above as a 1-dimensional extension of this one, we see that a converse of Proposition \ref{prop:ext-stab} cannot always hold, at least in the non-nilpotent case.

Next, let
\[ A := \frac{1}{\sqrt{3}} 
\begin{pmatrix} 0 & 1 & 1 \\
 1 & 0 & 1 \\
 1 & -1 & 0 \\
 1 & 1 & -1
\end{pmatrix}. \]
This gives a non-Einstein solvsoliton with $\lambda = -1$ and 
\begin{align*}
\Ric &= -\frac{1}{3} \mrm{diag}\left(2,4,2,3,3,3,3\right) \\
D    &= \frac{1}{3} \mrm{diag}\left(1,-1,1,0,0,0,0\right) 
\end{align*}
so that $\tr D = 1/3$.  Again, the condition of Corollary \ref{cor:alg-stab-cond} is not satisfied, since
\[ h := 
\begin{pmatrix}
 8 &   &   &   &   &   &   \\
   & -4 &   &   &   &   &   \\
   &   & 8 &   &   &   &   \\
   &   &   & -3 &   &   &   \\
   &   &   &   & 1 & 3 &   \\
   &   &   &   & 3 & -3 & -4 \\
   &   &   &   &   & -4 & 1
\end{pmatrix}
\quad \Longrightarrow \quad
\frac{Q(h)}{\langle h,h \rangle} = \frac{18}{107} > \frac{1}{6} = \frac{1}{2} \tr D.
\]
It is also easy to see that the sectional curvature has mixed sign, so that Proposition \ref{prop:sec-stab} does not apply.

Taking any 1-dimensional extension of this example will yield an Einstein metric as in Example \ref{ex:hyp-prod}, meaning it is possible to ``improve'' the spectrum of $Q$ under extensions.
\end{ex}

\section{Two-step nilsolitons}\label{sec:2step}

In this section, we prove that all two-step nilsolitons are strictly linearly stable.  The idea is to use Corollary \ref{cor:alg-stab-cond} and extra information from the two-step nilpotent structure.

\subsection{Construction}

Let $\mfr{n}$ be a two-step nilpotent Lie algebra.  Suppose that $\mfr{z} := [\mfr{n},\mfr{n}]$ has dimension $p$, and let $\mfr{v}$ be a complementary subspace of dimension $q$, so that $\mfr{n} = \mfr{v} \oplus \mfr{z}$ and $n := \dim \mfr{n} = p+q$.  The dimensions $p$ and $q$ must satisfy
\[ 1 \leq p \leq \frac{1}{2} q (q-1) = \dim \mfr{so}(\R,q). \] 

Suppose that $\mfr{n}$ has an inner product $\langle \cdot,\cdot \rangle$ such that $\mfr{v} \perp \mfr{z}$.  We obtain a homomorphism of $\mbb{R}$-algebras, $J : \mfr{z} \rightarrow \End(\mfr{v})$, which is defined implicitly by the relation
\begin{equation}\label{eq:j-map-def}
\langle J_Z(U),V \rangle = \langle Z,[U,V] \rangle 
\end{equation}
for all $U,V \in \mfr{v}$ and $Z \in \mfr{z}$.  Each map $J_Z \in \End(\mfr{v})$ is therefore skew-symmetric.  The map $J : \mfr{z} \rightarrow \End(\mfr{v})$ is called the \textit{J-map}.

\begin{ex}\label{ex:heis-def}
We call a two-step nilpotent metric Lie algebra $\mfr{n}$ a \textit{generalized Heisenberg algebra} if
\begin{equation}\label{eq:j-map-sq}
J_Z^2 = -|Z|^2 \, \id_\mfr{v}
\end{equation}
for all $Z \in \mfr{z}$, or equivalently if
\[ J_Y J_Z + J_Z J_Y = -2 \langle Y,Z \rangle \, \id_\mfr{v} \]
for all $Y,Z \in \mfr{z}$.  The unique simply connected Lie group $\mcl{N}$ with Lie algebra $\mfr{n}$ is then called a \textit{generalized Heisenberg group}.  These spaces were introduced by A.~Kaplan \cite{Kaplan1980,Kaplan1981}; see also \cite{Berndt1995}.
 
If $\eta$ is the quadratic form on $\mfr{z}$ induced by the inner product (that is, $\eta(Z) := |Z|^2$), then condition \eqref{eq:j-map-sq} implies that $J$ extends to a representation of the Clifford algebra $\mrm{Cl}(\mfr{z},\eta)$ on $\mfr{v}$.  It is not hard to see that the converse is true, namely, that every representation of a Clifford algebra corresponds to a generalized Heisenberg algebra.  The classification of Clifford algebras therefore provides a classification of generalized Heisenberg algebras.
\end{ex}

\begin{ex}\label{ex:free-2step-def}
Let $\mfr{g}$ be a semi-simple Lie algebra, and let $J : \mfr{g} \rightarrow \mfr{gl}(q,\R)$ be a faithful representation.  Then we can give $\mfr{n} := \R^q \oplus \mfr{g}$ the structure of a two-step nilpotent metric Lie algebra as follows.  First, give $\R^q$ an inner product such that each $J_X$ is skew-symmetric, give $\mfr{g}$ an $\ad(\mfr{g})$-invariant inner product, and declare that the summands $\R^q$ and $\mfr{g}$ are orthogonal.  Now, the bracket structure is determined as in equation \eqref{eq:j-map-def}.

A special case is the usual representation of $\mfr{so}(q,\R)$ on $\R^q$.  We call the resulting algebra the  \textit{free two-step nilpotent Lie algebra on $q$ generators} and write $\mfr{F}_2(q) := \R^q \oplus \mfr{so}(q,\R)$.  This means $p = \dim \mfr{so}(q,\R) = \half q(q-1)$.
\end{ex}

\begin{ex}
The Ricci tensor of a two-step nilpotent metric Lie algebra $\mfr{n}=\mfr{v} \oplus \mfr{z}$ is \textit{optimal} (in the sense of \cite{Eberlein2008}) if 
\[ \Rc|_{\mfr{v}\times\mfr{v}} = - a \, \id_\mfr{v}, \quad 
   \Rc|_{\mfr{v}\times\mfr{z}} = \Ric|_{\mfr{z}\times\mfr{v}} = 0, \quad
   \Rc|_{\mfr{z}\times\mfr{z}} = b \, \id_\mfr{z}, \]
for some $a,b > 0$.  It is easy to see that metrics with optimal Ricci tensors are nilsolitons.  In fact, using the remarks following Propositions 7.3 and 7.4 in \cite{Eberlein2008}, the constant $a$ can be eliminated:
\[ \Ric = \frac{b}{q} \left(
\begin{BMAT}{ccc.ccc}{ccc.ccc}
 &                     & & &                  &  \\
 & - 2p \, \id_\mfr{v} & & &                  &  \\
 &                     & & &                  &  \\
 &                     & & &                  &  \\
 &                     & & & q \, \id_\mfr{z} &  \\
 &                     & & &                  &  \\
\end{BMAT} \right),
\quad
D = \frac{b}{q} (2p+q) \left(
\begin{BMAT}{ccc.ccc}{ccc.ccc}
 &            & & &                 &  \\
 &  \, \id_\mfr{v}     & & &                 &  \\
 &            & & &                 &  \\
 &            & & &                 &  \\
 &            & & & 2 \, \id_\mfr{z} &  \\
 &            & & &                 &  \\
\end{BMAT} \right)\]
and $\lambda = -\frac{b}{q}(4p+q)$.%, $\tr D = \frac{b}{q}(2p+q)^2$.

The nilsoliton metrics on generalized Heisenberg algebras and free two-step nilpotent algebras all have optimal Ricci tensors; see Subsections \ref{subsec:3step-heis} and \ref{subsec:3step-free} for more information.  It turns out that the property of having optimal Ricci tensor is equivalent to being a nilsoliton with eigenvalue type $(1,2)$, which means that the eigenvalues of the nilsoliton derivation are $1$ and $2$ (perhaps after rescaling).  This property is ``generic'' among two-step nilpotent metric Lie algebras in a specific sense; see \cite[Proposition 7.9]{Eberlein2008} or \cite[Theorem 4]{Nikolayevsky2011}. 
\end{ex}

\subsection{Curvature}

We use the results of Eberlein \cite{Eberlein1994} to describe the Riemann curvature tensor on a two-step nilpotent Lie group $\mfr{n} = \mfr{v} \oplus \mfr{z}$.  Let 
\[ \{U_1,\dots,U_n,Z_1,\dots,Z_m \} \]
be an orthonormal basis for $\mfr{n}$, with $U_i \in \mfr{v}$ and $Z_\alpha \in \mfr{z}$.  Now we will use Roman indices for elements of $\mfr{n}$ and Greek indices for elements of $\mfr{z}$.  The (non-zero) Riemann, sectional, and Ricci curvatures are 
\begin{align*}
R(U,V,W,W') &= \half \langle [U,V],[W,W'] \rangle 
               +\smfrac{1}{4} \langle [U,W],[V,W'] \rangle 
               -\smfrac{1}{4} \langle [U,W'],[V,W] \rangle \\
R(U,V,X,Y)  &= \smfrac{1}{4} \langle J_X U, J_Y V \rangle
               - \smfrac{1}{4} \langle J_X V, J_Y U \rangle \\
R(U,X,V,Y) &= -\smfrac{1}{4} \langle J_X V, J_Y U \rangle 
\end{align*}
\begin{align*}
\sec(U,V) &= -\smfrac{3}{4} |[U,V]|^2 & & & \Rc(U,V) &= \smfrac{1}{2} \sum_\alpha \langle J_\alpha^2 U,V \rangle \\
\sec(U,X) &= \smfrac{1}{4} |J_X U|^2  & & & \Rc(X,Y) &= - \smfrac{1}{4} \tr (J_X \circ J_Y)
\end{align*}
for all $U,V,W,W' \in \mfr{v}$ and $X,Y \in \mfr{z}$.  In terms of structure constants, the Ricci curvatures are
\[ R_{ij} = \smfrac{1}{2} \sum_{k\alpha} c_{ik}^\alpha c_{kj}^\alpha, \quad
   R_{\alpha\beta} = \smfrac{1}{4} \sum_{ij} c_{ij}^\alpha c_{ij}^\beta. \]
and, recalling that $\Ric$ is negative-definite on $\mfr{v}$ and positive-definite on $\mfr{z}$, we define for $V \in \mfr{v}$ and $Z \in \mfr{z}$,
\[ \rho_{-} := - \min_{|V|=1} \Rc(V,V), \quad
   \rho_{+} :=   \max_{|Z|=1} \Rc(Z,Z). \]
 
\begin{lem}\label{lem:2step-ric-est}
When $\mfr{n}$ is a two-step nilsoliton of type $(p,q)$, rescaled so that $\scal = -1$, we have
\begin{align}
\rho_-,\rho_+ &\leq |\Ric|^2 \label{eq:ric-evals} \\
|\Ric|^2 &\geq \frac{1}{p} + \frac{4}{q}. \label{eq:norm-ric-est}
\end{align}
\end{lem}

\begin{proof}
Let $G$ be a simply-connected Lie group with a left-invariant metric.  As the group acts by isometries on itself, all of the geometry is encoded in a single tangent space with inner product.  We take the Lie algebra $\mathfrak g = \Lie \  G \simeq T_eG$ of $G$ to be our tangent space of interest and write it as $(\mathbb R^n, \mu)$, where $\mu$ denotes the Lie bracket on the underlying vector space $\mathbb R^n$.  We denote the inner product on this vector space by $\ip{\cdot,\cdot}$.  (We choose our identification of $\mathfrak g$ with  $\mathbb R^n$ so that the `standard basis' is orthonormal relative to our fixed inner product on $\mathfrak g$.)

We study $G$ and $\mathfrak g$ by viewing the Lie bracket $\mu$ as an element of the space
\[ V := \wedge ^2 (\mathbb R^n)^*\otimes \mathbb R^n \]
 of (anti-symmetric) algebra structures on $\mathbb R^n$. This perspective has proven very effective for working with nilsolitons, cf.~\cite{Lauret2010-std,Lauret2011-sol,Jablonski2011-existence}.
 For the sake of consistency and ease of referencing, we hold closely  to the notation used by Lauret.

Take $\{e_i\}$ to be the standard basis of $\mathbb R^n$ and denote by $e_i'$ the element of $(\mathbb R^n)^*$ which is dual to $e_i$.  Then
\[ \{ v_{ijk} := (e_i' \wedge e_j')\otimes e_k \ | \ 1 \leq i<j\leq n, 1\leq k\leq n \} \]
is a basis  for $V$.  The element $v_{ijk}$ is the element such that, when evaluated on basis vectors of $\mathbb R^n$, we have $v_{ijk}(e_i,e_j) = e_k = - v_{ijk}(e_j,e_i)$ and zero otherwise.  We  define an inner product on $V$ such that the above basis is orthonormal.  

Recall that the vector space $V$ comes naturally equipped with an action $\pi$ of $\mathfrak{gl}(n,\mathbb R)$ defined by 
\[ (\pi(X) \mu )(v,w) = X\mu(v,w) - \mu(Xv,w) - \mu(v,Xw) \]
where $X\in \mathfrak{gl}(n,\mathbb R)$, $\mu\in V$, and $v,w\in\mathbb R^n$.  As we will see, the basis $\{v_{ijk}\}$ of $V$ actually consists of weight vectors for this representation.

In the following, we equip $\mathfrak{gl}(n,\mathbb R)$ with the inner product $\ip{\alpha,\beta} := \tr  \alpha\beta^*$.  Then for a diagonal matrix $\alpha = \mrm{diag}(a_1,\dots, a_n)$, we have
	\begin{equation}\label{eqn: weights of the rep}
	\pi (\alpha) v_{ijk} = (-a_i-a_j+a_k) v_{ijk} = \ip{\alpha, \alpha_{ij}^k} v_{ijk}
	\end{equation}
where $\alpha_{ij}^k := -E_{ii}-E_{jj} + E_{kk}$ and $E_{ii}$ is the matrix with a 1 in the $i^{th}$ position along the diagonal and zero elsewhere.  Thus the diagonal matrices $\alpha_{ij}^k$ are the weights of the representation (relative to the inner product that we have fixed on $\mathfrak{gl}(n,\mathbb R)$).

For a given element $\mu \in V$, we write $\mu = \sum \mu_{ij}^k v_{ijk}$.  When $\mu$ is a Lie algebra structure on $\mathbb R^n$, we see that the $\mu_{ij}^k$ are the usual structure constants which satisfy $\mu(e_i,e_j) = \sum_k \mu_{ij}^k e_k$.

Let $\mu \in V$ be a nilpotent Lie algebra structure on $\mathbb R^n$.  As discussed above, the metric Lie algebra $(\mathbb R^n,\mu, \langle \cdot,\cdot \rangle )$ corresponds to a simply-connected Lie group with left-invariant metric which we denote by $G_\mu$.  Denote the Ricci and scalar curvatures of $G_\mu$ by $\Ric_\mu$ and $\scal(\mu)$, respectively.  In terms of the norm on $V$ given above, we have
\[ \scal(\mu) = -\frac{1}{4} |\mu|^2 \]
and, further,
\[ m(\mu) = 4 \Ric_\mu \]
where $m: V \to \mathfrak{gl}(n,\mathbb R)$ is the moment map of the representation, which is defined by
\[ \ip{m(\mu),X} = \ip{\pi(X)\mu,\mu} \]
for all $X\in \mathfrak{gl}(n,\mathbb R)$. 

Now let $G_\mu$ be a nilsoliton, so that $\Ric_\mu = \lambda \, \id +D$ for some $\lambda\in\mathbb R$ and $D$ a derivation of the Lie algebra $(\mathbb R^n,\mu)$.  First observe that $\mu$ is an eigenvector of $\Ric_\mu$ with eigenvalue $-\lambda$.  (This follows from the fact that $\pi(D)\mu = 0$, as $D$ is a derivation, and $\pi(\id)$ acts by $-1$.)  We may assume that the Ricci tensor is diagonal with respect to $\{e_i\}$ and, encoding the previous statement using the weights of our representation, we have
\[ \ip{\Ric_\mu, \alpha_{ij}^k} = -\lambda \quad \mbox{ for } (i,j,k) \mbox{ such that } \mu_{ij}^k \not = 0 \]
Now applying the definition of the moment map given above, we also see that
\begin{align*}
16 |\Ric_\mu|^2 
&= \ip{4 \pi(\Ric_\mu) \mu,\mu}\\
&= \ip{4\lambda \pi(\id)\mu,\mu }\\
&= -4\lambda |\mu|^2
\end{align*}
from which we obtain $-\lambda = \frac{4 |\Ric_\mu |^2}{ |\mu|^2}$.  If we normalize so that $\scal(\mu) = -1$, we see that $|\mu|^2=4$ and so $-\lambda = |\Ric_\mu|^2$.  This yields
\[ \ip{\Ric_\mu, \alpha_{ij}^k} = |\Ric_\mu|^2   \quad \mbox{ for } (i,j,k) \mbox{ such that } \mu_{ij}^k \not = 0 \]

Finally we consider the case of interest, namely when $\mu$ is 2-step nilpotent.  From \cite{Eberlein1994}, we know $\Ric_\mu$ preserves the commutator subalgebra $\mu(\mathbb R^n,\mathbb R^n)$ (and also its orthogonal complement).  Furthermore,  the eigenvalues of $\Ric_\mu$ on the commutator subalgebra $\mu(\mathbb R^n,\mathbb R^n)$ are positive and non-positive on its orthogonal compliment.  Applying \eqref{eqn: weights of the rep} we see that 
\[ |\rho| \leq |\Ric_\mu|^2 \]
for any eigenvalue $\rho$ of $\Ric_\mu$ and so we have \eqref{eq:ric-evals}.

For \eqref{eq:norm-ric-est}, recall that for any nilsoliton, we have $\ip{\Ric_\mu , \alpha_{ij}^k }  = |\Ric_\mu|^2 $ for $(i,j,k)$ such that $\mu_{ij}^k \not =0$.  As the $\alpha_{ij}^k$ are all of the same length, this implies that $\Ric_\mu$ is the minimal convex combination of $\{ \alpha_{ij}^k \ | \ \mu_{ij}^k \not = 0\}$, i.e. the element of minimal length in the convex combination.

From this, we see that $|\Ric_\mu|^2$ is greater than or equal to the minimal convex combination of all the $\alpha_{ij}^k$ such that $1\leq i,j\leq q$ and $q+1 \leq k \leq n=q+p$.  We denote said minimal convex combination by $\alpha$ and observe that it satisfies
\begin{equation}\label{eqn:mcc}
	\ip{\alpha,\alpha_{ij}^k } = |\alpha|^2
\end{equation}
for all $(i,j,k)$ such that $1\leq i,j\leq q$ and $q+1 \leq k \leq n=q+p$.  By fixing any two of $i,j,k$ and varying the third, and applying \eqref{eqn: weights of the rep}, we see that
\[ \alpha = \mrm{diag} ( \underbrace{ a, \dots, a}_{q}, \underbrace{ b, \dots b}_{p}) \]
Now \eqref{eqn:mcc} implies
\[ qa^2 + pb^2 = -2a+b \]
As $\alpha$ is a convex combination of the $\alpha_{ij}^k$, and each $\alpha_{ij}^k$ has trace $-1$, we have
\[ -1 = qa+pb \]
From these two equations, we see that 
\[ b = \frac{q+\sqrt{q^2+4np}}{2np}
= \frac{q+\sqrt{(2p+q)^2}}{2np}
= \frac{1}{p} \]
and thus $-a = 2/q$.  This means
\[ |\Ric_\mu|^2 
\geq |\alpha|^2 
= -2a+b 
= \frac{4}{q} + \frac{1}{p}. \qedhere \]
\end{proof}

\subsection{Estimates}   
   
Our goal is to estimate the eigenvalues of the operator $Q(h) = \langle \Ro h + \Ric \circ h,h \rangle$.  Namely, according to Corollary \ref{cor:alg-stab-cond}, we want to show that $Q(h) < \half \tr D |h|^2$, where $h$ is any symmetric 2-tensor.  These tensors have the same form as in \eqref{eq:h-decomp}.

For $Q$, we will consider the Riemann and Ricci parts separately.  A computation similar to \eqref{eq:Ro-expand} shows that
\[ \langle \Ro h,h \rangle 
= \sum_{ijk\ell} R_{ijk\ell} h_{i\ell} h_{jk}
 + 2 \sum_{ij\alpha\beta} (R_{ij\alpha\beta}+R_{i\alpha j\beta}) h_{i\beta} h_{j\alpha} 
 - 2 \sum_{ij\alpha\beta} R_{i\alpha j\beta} h_{ij} h_{\alpha\beta}. \]
Assume that the orthonormal basis of $\mfr{n}$ is chosen so that the variation $h$ is diagonal on $\mfr{v} \times \mfr{v}$ and on $\mfr{z} \times \mfr{z}$. (We cannot assume that $h$ has any special form on $\mfr{v} \times \mfr{z}$.)  Then the first term is
\begin{equation}\label{eq:2step-ijkl-est}
\begin{aligned}
\sum_{ijk\ell} R_{ijk\ell} h_{i\ell} h_{jk}
&= \sum_{ij} R_{ijji} h_{ii} h_{jj} 
%&\leq \sum_{ij} |R_{ijji}| h_{ii}^2 \\
\leq \smfrac{3}{4} \sum_{i} h_{ii}^2 \sum_j |[U_i,U_j]|^2 \\
&= \smfrac{3}{4} \sum_{i} h_{ii}^2 \sum_{j\alpha} (c_{ij}^\alpha)^2 
= \smfrac{3}{4} \sum_{i} h_{ii}^2 \, 2 |R_{ii}| \\
&\leq \smfrac{3}{2} \rho_{-} \sum_{i} h_{ii}^2
\end{aligned}
\end{equation}
where the first inequality uses the arithmetic/geometric mean inequality.  The middle term is
\begin{equation}\label{eq:2step-mixed-est}
\begin{aligned}
2 \sum_{ij\alpha\beta} (R_{ij\alpha\beta} + R_{i\alpha j\beta}) h_{i\beta} h_{j\alpha} 
&= \sum_{ij\alpha\beta} \left( -\langle J_\beta U_i, J_\alpha U_j \rangle  + \half \langle J_\alpha U_i, J_\beta U_j \rangle \right) h_{i\beta} h_{j\alpha} \\
&= - \sum_{ijk\alpha\beta} c_{ik}^\beta c_{jk}^\alpha h_{i\beta} h_{j\alpha}
 + \half \sum_{ijk\alpha\beta} c_{ik}^\alpha c_{jk}^\beta h_{i\beta} h_{j\alpha} \\
&\leq - \sum_{k} \sum_{i\beta} c_{ik}^\beta h_{i\beta} \sum_{j\alpha} c_{jk}^\alpha  h_{j\alpha}
 + \smfrac{1}{4} \sum_{ijk\alpha\beta} |c_{ik}^\alpha c_{jk}^\beta| ( h_{i\beta}^2 + h_{j\alpha}^2) \\   &= - \sum_{k} \left( \sum_{i\beta} c_{ik}^\beta h_{i\beta} \right)^2 
 + \smfrac{1}{2} \sum_{ijk\alpha\beta} |c_{ik}^\alpha c_{jk}^\beta| h_{i\beta}^2 \\   
&\leq \smfrac{1}{4} \sum_{i\beta} h_{i\beta}^2 \sum_{jk\alpha} \big( (c_{ik}^\alpha)^2 + (c_{jk}^\beta)^2 \big) \\   
&= \smfrac{1}{4} \sum_{i\beta} h_{i\beta}^2 \left( \sum_{j} 2|R_{ii}| + \sum_{\alpha} 4|R_{\beta\beta}| \right) \\
&\leq \left( \half q \rho_{-} + p \rho_{+} \right) \sum_{i\beta} h_{i\beta}^2 
\end{aligned}
\end{equation}
and last term is
\begin{equation}\label{eq:2step-ijab-est}
\begin{aligned}
-2 \sum_{ij\alpha\beta} R_{i\alpha j\beta} h_{ij} h_{\alpha\beta}
&= 2 \sum_{i\alpha} R_{i\alpha\alpha i} h_{ii} h_{\alpha\alpha} \\
&\leq \sum_{i\alpha} |R_{i\alpha\alpha i}| (h_{ii}^2 + h_{\alpha\alpha}^2) \\
&= \sum_{i\alpha} |R_{i\alpha\alpha i}| h_{ii}^2 
    + \sum_{i\alpha} |R_{i\alpha\alpha i}| h_{\alpha\alpha}^2 \\
&= \smfrac{1}{4} \sum_{i} h_{ii}^2 \sum_{j\alpha} (c_{ij}^\alpha)^2 
    + \smfrac{1}{4} \sum_{\alpha} h_{\alpha\alpha}^2 \sum_{ij} (c_{ij}^\alpha)^2 \\
&= \smfrac{1}{4} \sum_{i} h_{ii}^2 \, 2|R_{ii}| 
   + \smfrac{1}{4} \sum_{\alpha} h_{\alpha\alpha}^2 \, 4|R_{\alpha\alpha}| \\
&\leq \smfrac{1}{2} \rho_{-} \sum_{i} h_{ii}^2
   + \rho_{+} \sum_\alpha h_{\alpha\alpha}^2.
\end{aligned}
\end{equation}

For the Ricci part of $Q$, we assume that the orthonormal basis of $\mfr{n}$ is chosen such that the Ricci operator is diagonal.  Then, as in \eqref{eq:Rico-expand}, 
\begin{align*}
\langle \Ric \circ h,h \rangle
&= \sum_{ijk} R_{ij} h_{jk} h_{ki}
+ \sum_{ij\alpha} R_{ij} h_{j\alpha} h_{\alpha i}
+ \sum_{i\alpha\beta} R_{\alpha\beta} h_{\beta i} h_{i \alpha}
+ \sum_{\alpha\beta\gamma} R_{\alpha\beta} h_{\beta\gamma} h_{\gamma\alpha} \\
&= \sum_{ij} R_{ii} h_{ij}^2
  +\sum_{i\alpha} R_{ii} h_{i\alpha}^2
  +\sum_{i\alpha} R_{\alpha\alpha} h_{i\alpha}^2
  +\sum_{\alpha\beta} R_{\alpha\alpha} h_{\alpha\beta}^2 \\
&\leq \rho_{+} \sum_{i\alpha} h_{i\alpha}^2
  + \rho_{+} \sum_{\alpha\beta} h_{\alpha\beta}^2
\end{align*}
Noticing that the right-hand expression in \eqref{eq:h-norm-n} is invariant under change of orthonormal basis, we combine the above calculations to obtain
\[ Q(h)
\leq 2\rho_{-} \sum_{ij} h_{ij}^2 
+ \left( \frac{1}{2} q \rho_{-} + (p+1) \rho_{+} \right) \sum_{i\beta} h_{i\beta}^2
+ 2\rho_{+} \sum_{\alpha\beta} h_{\alpha\beta}^2. \]

We now have another criterion for stability of nilsolitons in the two-step case.

\begin{prop}\label{prop:2step-stab-cond}
A two-step nilsoliton is strictly linearly stable if all of the following hold.
\begin{align*}
\rho_{-}, \rho_{+} &< \frac{1}{4} \tr D \\
\frac{1}{2} q \rho_{-} + (p + 1) \rho_{+} &< \tr D
\end{align*}
\end{prop}

We now use this to prove the main result of this section.

\begin{prop}\label{prop:2step-stab}
Every two-step nilsoliton is strictly linearly stable.
\end{prop}

\begin{proof} 
We will use Proposition \ref{prop:2step-stab-cond} and the results of Lemma \ref{lem:2step-ric-est}.  As in the Lemma, we normalize so that $\scal=-1$.  First, we have
\[ \lambda = \frac{\tr \Ric^2}{\tr \Ric} = \frac{|\Ric|^2}{\scal} = - |\Ric|^2, \]
and we can rewrite the trace of $D$ as 
\[ \tr D = \scal - \lambda (p+q) = (p+q) |\Ric|^2 - 1. \]
Also, \eqref{eq:norm-ric-est} implies that
\begin{equation}\label{eq:norm-ric-est2}
- \frac{1}{\frac{1}{p}+\frac{4}{q}} \leq -\frac{1}{|\Ric|^2}. 
\end{equation}

By \eqref{eq:ric-evals}, for the first two conditions in Proposition \ref{prop:2step-stab-cond} to be satisfied it is enough that
\[ |\Ric|^2 \stackrel{?}{<} \frac{1}{4}(p+q) |\Ric|^2 - \frac{1}{4}. \]
Using \eqref{eq:norm-ric-est2}, it is therefore enough to show that
\begin{equation}\label{eq:2step-cond1}
1 \stackrel{?}{<} \frac{(2p+q)^2}{4(4p+q)}
= \frac{1}{4}(p+q) - \frac{1}{4} \frac{1}{\frac{1}{p}+\frac{4}{q}}
\leq \frac{1}{4}(p+q) - \frac{1}{4 |\Ric|^2}.
\end{equation}
Similarly, for third condition in Proposition \ref{prop:2step-stab-cond} to be satisfied it is enough that
\[ \frac{1}{2} q |\Ric|^2 + (p+1) |\Ric|^2 \stackrel{?}{<} (p+q) |\Ric|^2 - 1. \]
Again using \eqref{eq:norm-ric-est2}, it is enough to show that
\begin{equation}\label{eq:2step-cond2}
1 \stackrel{?}{<} \frac{q(2p+q)}{2(4p+q)}
= \frac{1}{2} q - \frac{1}{\frac{1}{p}+\frac{4}{q}}
\leq \frac{1}{2} q - \frac{1}{|\Ric|^2}. 
\end{equation}
One other condition is the restriction on the dimensions of $\mfr{v}$ and $\mfr{z}$ from before,
\begin{equation}\label{eq:2step-cond3}
1 \leq p \stackrel{?}{\leq} \frac{1}{2} q(q-1),
\end{equation}
which must hold for any two-step nilpotent Lie algebra.  As $p\geq 1$, one can rewrite conditions \eqref{eq:2step-cond1}, \eqref{eq:2step-cond2}, and \eqref{eq:2step-cond3} as
\begin{align*}
q &\stackrel{?}{>} 2(1-p+\sqrt{1+2p}) \\
q &\stackrel{?}{>} 1-p+\sqrt{1+6p+p^2} \\
q &\stackrel{?}{\geq} \frac{1}{2}(1+\sqrt{1+8p})
\end{align*}
The curves from the first two conditions intersect at $p=3/2$, and the second and third intersect at $p = (4+3\sqrt{2})/2 \cong 4.1$; see Figure \ref{fig:2step}, where the solid line represents the third inequality.  From this, it is not difficult to see that all pairs $(p,q)$ satisfy these three conditions simultaneously, except those of the form $(1,2)$, $(1,3)$, and $(2,3)$.  But these satisfy $p+q \leq 5$ and are covered in Section \ref{sec:low-dim}.    
\end{proof}

\begin{figure}[t]
\includegraphics[width=0.8\textwidth]{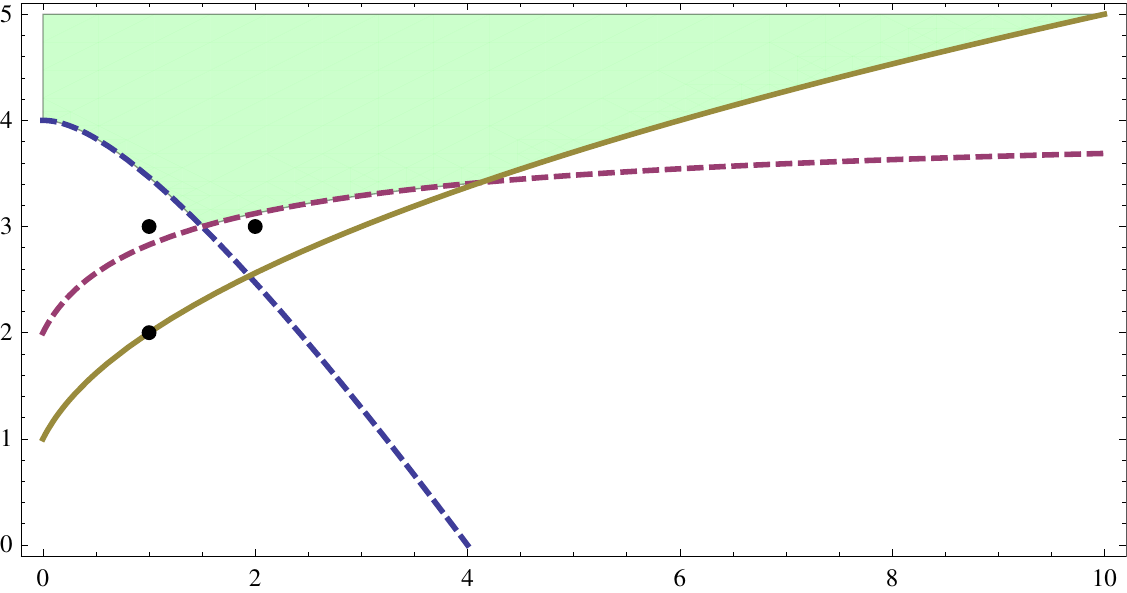}
\caption{Region in the $(p,q)$-plane corresponding to conditions \eqref{eq:2step-cond1}, \eqref{eq:2step-cond2}, and \eqref{eq:2step-cond3}, and the three types $(1,2)$, $(1,3)$, $(2,3)$ not satisfying those conditions.}
\label{fig:2step}
\end{figure}

\section{Three-step solvsolitons}\label{sec:3step}

Let $\mfr{s}$ be a three-step solvable Lie algebra.  That is, $\mfr{n} := [\mfr{s},\mfr{s}] = \mfr{v} \oplus \mfr{z}$ is two-step nilpotent, and $\mfr{s} = \mfr{n} \rtimes \mfr{a}$ where $\mfr{a}$ is abelian.  Assume that $\dim \mfr{n} = n = p + q$ (as in the previous section) and that $\mfr{a}$ is one-dimensional (as in the construction following Theorem \ref{thm:lauret}).  The goal of this section is to prove that Einstein metrics on certain spaces of this form are linearly stable.  

\begin{prop}\label{prop:stab-solv-heis-free}
Let $\mfr{s}$ be a solvable Lie algebra whose nilradical is codimension-one, and either a generalized Heisenberg algebra or a free two-step nilpotent Lie algebra.  Then the natural Einstein metric on $\mfr{s}$ is strictly linearly stable.
\end{prop}

The idea is to use Corollary \ref{cor:alg-stab-cond}, which means we must understand the operator $\Ro$ on the relevant spaces.  The analysis similar in spirit to that in the previous section, but more complicated due to the added complexity of the solvable spaces involved.

\subsection{Curvature}

We assume that the metric on the nilradical $\mfr{n}$ is a nilsoliton with eigenvalue type $(1,2)$ (i.e., optimal Ricci tensor), and set $c := b/n$.

Let 
\[ \{U_1,\dots,U_n,Z_1,\dots,Z_m,A \} \]
be an orthonormal basis for $\mfr{s}$, with $U_i \in \mfr{v}$, $Z_\alpha \in \mfr{z}$, and $A \in \mfr{a}$.  We will use capital Roman indices for elements of $\mfr{n}$, lowercase Roman indices for elements of $\mfr{v}$, Greek indices for elements of $\mfr{z}$, and 0 as the index for $\mfr{a}$.

With respect to this basis, let us describe the Riemann curvature.  Writing $J_\alpha$ to mean $J_{Z_\alpha}$, here are all the curvature components sorted according to the number of indices from $\mfr{v}$, $\mfr{z}$, and $\mfr{a}$.

\begin{center}
\begin{longtable}{llcll}
4-0-0 & \multicolumn{4}{l}{$R_{ijk\ell} = \half \langle [U_i,U_j],[U_k,U_\ell]\rangle 
             + \smfrac{1}{4} \langle [U_i,U_k],[U_j,U_\ell]\rangle 
             - \smfrac{1}{4} \langle [U_i,U_\ell],[U_j,U_k]\rangle$} \\[6pt]
      & \multicolumn{4}{l}{$\qquad \qquad + c \, \delta_{ik} \delta_{j\ell} 
                                  - c \, \delta_{i\ell} \delta_{jk}$} \\
      & & & & \\
3-1-0 & $R_{ijk\alpha} = 0$  & \qquad & 2-2-0 & $R_{ij\alpha\beta} = -\smfrac{1}{4} \langle J_\beta U_i, J_\alpha U_j \rangle + \smfrac{1}{4} \langle J_\alpha U_i, J_\beta U_j \rangle$ \\
3-0-1 & $R_{ijk0} = 0$ &        & 2-2-0 & $R_{i\alpha j\beta}= -\smfrac{1}{4} \langle J_\alpha U_j, J_\beta U_i \rangle + 2c \, \delta_{ij} \delta_{\alpha\beta}$ \\
      &                &        & 2-1-1 & $R_{ij\alpha0}     = - \sqrt{c} \, \langle J_\alpha U_i, U_j \rangle$ \\   
      &                &        & 2-1-1 & $R_{i\alpha j0}     = - \half \sqrt{c} \, \langle J_\alpha U_i, U_j \rangle$ \\   
      &                &        & 2-0-2 & $R_{i0j0}     = c \, \delta_{ij}$ \\
      &                &        & 2-0-2 & $R_{ij00}     = 0$ \\
      &                &        &       & \\      
1-3-0 & $R_{i\alpha\beta\gamma} = 0$ & \qquad & 0-4-0 & $R_{\alpha\beta\gamma\delta} 
                                                       = 4c \, \delta_{\alpha \gamma} \delta_{\beta \delta} 
                                                       - 4c \, \delta_{\alpha \delta} \delta_{\beta \gamma}$ \\
1-2-1 & $R_{i\alpha\beta 0} = R_{\alpha\beta i0} = 0$  &        & 0-3-1 & $R_{\alpha\beta\gamma 0} = 0 $ \\
1-1-2 & $R_{i\alpha 00} = R_{i0\alpha 0} = 0$          &        & 0-2-2 & $R_{\alpha 0 \beta 0} = 																																												4c \, \delta_{\alpha\beta}$ \\   
1-0-3 & $R_{i000} = 0 $                                &        & 0-2-2 & $R_{\alpha\beta 00} = 0$ \\   
      &                                                &        & 0-1-3 & $R_{\alpha 000} = 0$ \\
      &                                                &        & 0-0-4 & $R_{0000} = 0$
\end{longtable}
\end{center}

\subsection{Estimates}

In this section we estimate $\langle \Ro h,h \rangle$ for symmetric 2-tensors $h$.  First we note that, with respect to our orthonormal basis above, such a tensor has a representation similar to \eqref{eq:h-decomp}.  We again make the further assumption that the basis of $\mfr{s}$ is chosen so that $h$ is diagonal on $\mfr{v} \times \mfr{v}$ and $\mfr{z} \times \mfr{z}$.  In this case, the norm with respect to the inner product on $\mfr{s}$ is
\begin{equation}\label{eq:h-norm}
|h|^2
= \sum_{i} h_{ii}^2 + 2 \sum_{i\beta} h_{i\beta}^2 + 2 \sum_i h_{i0}^2 
+ \sum_{\alpha} h_{\alpha\alpha}^2 + 2 \sum_{\alpha} h_{\alpha 0}^2 + h_{00}^2.
\end{equation}

We first eliminate any terms in $\langle \Ro h,h \rangle$ that vanish due to a 0 index in a way similar to \eqref{eq:Ro-expand}.
\begin{align*}
\langle \Ro h,h \rangle
&=  R_{IJKL} h_{IL} h_{JK} + R_{IJK0} h_{I0} h_{JK} 
  + R_{IJ0L} h_{IL} h_{J0} + \cancel{R_{IJ00}} h_{I0} h_{J0} \\
&\quad + R_{I0KL} h_{IL} h_{0K} + R_{I0K0} h_{I0} h_{0K} 
       + R_{I00L} h_{IL} h_{00} + \cancel{R_{I000}} h_{I0} h_{00} \\
&\quad + R_{0JKL} h_{0L} h_{JK} + R_{0JK0} h_{00} h_{JK} 
       + R_{0J0L} h_{0L} h_{J0} + \cancel{R_{0J00}} h_{00} h_{J0} \\
&\quad + \cancel{R_{00KL}} h_{0L} h_{0K} + \cancel{R_{00K0}} h_{00} h_{0K} 
       + \cancel{R_{000L}} h_{0L} h_{00} + \cancel{R_{0000}} h_{00} h_{00} \\
&=  \underbrace{R_{IJKL} h_{IL} h_{JK}}_{I} 
 +  \underbrace{4 R_{IJK0} h_{I0} h_{JK}}_{II} 
 +  \underbrace{2 R_{I0J0} h_{I0} h_{J0}}_{III} 
 +  \underbrace{(- 2) R_{I0J0} h_{IJ} h_{00}}_{IV} . 
\end{align*}
We analyze these four terms separately.   Our method is to expand once again as in \eqref{eq:Ro-expand}, cancel any zero curvature terms, and insert the non-zero curvature terms.  We then use weighted arithmetic/geometric mean inequalities (with constants $A,B,C,D \in \R$)\footnote{The constants $A$ and $D$ should not be confused with the basis element $A$ or the soliton derivation $D$.} to obtain expressions comparable to $|h|^2$.
\begin{align*}
I
&= R_{IJKL} h_{IL} h_{JK} \\
&= R_{ijk\ell} h_{i\ell} h_{jk} 
   + \cancel{R_{ijk\delta}} h_{i\delta} h_{jk} 
   + \cancel{R_{ij\gamma\ell}} h_{i\ell} h_{j\gamma}
   + R_{ij\gamma\delta} h_{i\delta} h_{j\gamma} \\
&\quad + \cancel{R_{i\beta k\ell}} h_{i\ell} h_{\beta k} 
       + R_{i\beta k\delta} h_{i\delta} h_{\beta k} 
       + R_{i\beta\gamma\ell} h_{i\ell} h_{\beta\gamma}
       + \cancel{R_{\alpha j \gamma \delta}} h_{i\delta} h_{\beta\gamma} \\
&\quad + \cancel{R_{\alpha jk\ell}} h_{\alpha\ell} h_{jk} 
       + R_{\alpha jk\delta} h_{\alpha\delta} h_{jk} 
       + R_{\alpha j\gamma \ell} h_{\alpha\ell} h_{j \gamma} 
       + \cancel{R_{\alpha j \gamma \delta}} h_{\alpha \delta} h_{j \gamma} \\
&\quad + R_{\alpha\beta k\ell} h_{\alpha\ell} h_{\beta k} 
       + \cancel{R_{\alpha\beta k\delta}} h_{\alpha\delta} h_{\beta k} 
       + \cancel{R_{\alpha\beta\gamma\ell}} h_{\alpha\ell} h_{\beta\gamma} 
       + R_{\alpha\beta\gamma\delta} h_{\alpha\delta} h_{\beta\gamma} \\
&=     R_{ijk\ell} h_{i\ell} h_{jk}
   + 2 R_{ij\alpha\beta} h_{i\beta} h_{j\alpha}
   + 2 R_{i\alpha j\beta} h_{i\beta} h_{j\alpha}
   - 2 R_{i\beta j\alpha} h_{ij} h_{\alpha\beta}
   +   R_{\alpha\beta\gamma\delta} h_{\alpha\delta} h_{\beta\gamma} \\
&= \left( \half \langle [U_i,U_j],[U_k,U_\ell]\rangle 
          + \smfrac{1}{4} \langle [U_i,U_k],[U_j,U_\ell]\rangle 
          - \smfrac{1}{4} \langle [U_i,U_\ell],[U_j,U_k]\rangle \right. \\
&\qquad  \left. + c \, \delta_{ik} \delta_{j\ell} 
                - c \, \delta_{i\ell} \delta_{jk} \right) h_{i\ell} h_{jk} \\
&\quad   + \left( -\half \langle J_\beta U_i, J_\alpha U_j \rangle  
         + \half \langle J_\alpha U_i, J_\beta U_j \rangle 
         - \half \langle J_\alpha U_j, J_\beta U_i \rangle 
         \right) h_{i\beta} h_{j\alpha} 
         + 4c \, \delta_{ij} \delta_{\alpha\beta} h_{i\beta} h_{j\alpha} \\
&\quad  + \half \langle J_\alpha U_j, J_\beta U_i \rangle h_{ij} h_{\alpha\beta}
  - 4c \, \delta_{ij} \delta_{\alpha\beta} h_{ij} h_{\alpha\beta} 
  + 4c (\delta_{\alpha\gamma} \delta_{\beta\delta} - \delta_{\alpha \delta} \delta_{\beta\gamma}) h_{\alpha\delta} h_{\beta\gamma}  \\
&= \smfrac{3}{4} \langle [U_i,U_j],[U_k,U_\ell] \rangle h_{i\ell} h_{jk} \\
&\quad  - \langle J_\beta U_i, J_\alpha U_j \rangle  h_{i\beta} h_{j\alpha}
        + \half \langle J_\alpha U_i, J_\beta U_j \rangle h_{i\beta} h_{j\alpha} 
        + \half \langle J_\alpha U_j, J_\beta U_i \rangle h_{ij} h_{\alpha\beta} \\
&\quad + c (h_{ij}^2 - h_{ii} h_{jj})
       + 4c (h_{i\alpha}^2 - h_{ii} h_{\alpha\alpha})
       + 4c (h_{\alpha \beta}^2 - h_{\alpha\alpha} h_{\beta\beta})
\end{align*}
We estimate this term-by-term, and we now use explicit sums for clarity.  The first is handled in way similar to \eqref{eq:2step-ijkl-est}.
\[ \smfrac{3}{4} \sum_{ijk\ell} \langle [U_i,U_j],[U_k,U_\ell] \rangle h_{i\ell} h_{jk}
\leq 3p \smfrac{b}{q} \sum_i h_{ii}^2 \] 
Unfortunately, the idea of \eqref{eq:2step-mixed-est} does not yield a sufficient estimate for the second and third terms.
\begin{align*}
&-\sum_{ij\alpha\beta} \langle J_\beta U_i, J_\alpha U_j \rangle  h_{i\beta} h_{j\alpha}  
   + \half \sum_{ij\alpha\beta} \langle J_\alpha U_i, J_\beta U_j \rangle h_{i\beta} h_{j\alpha} \\
&\leq \half \sum_{ij\alpha\beta}|\langle J_\beta U_i, J_\alpha U_j \rangle| (h_{i\beta}^2 + h_{j\alpha}^2) 
            + \smfrac{1}{4} \sum_{ij\alpha\beta} |\langle J_\alpha U_i, J_\beta U_j \rangle| (h_{i\beta}^2 + h_{j\alpha}^2) \nonumber \\
&= \left( \sum_{j\alpha} |\langle J_\beta U_i, J_\alpha U_j \rangle| 
         + \half \sum_{j\alpha} |\langle J_\alpha U_i, J_\beta U_j \rangle| \right) 
      \sum_{i\beta} h_{i\beta}^2 
\end{align*}
As in \eqref{eq:2step-ijab-est}, we estimate the fourth term.
\[ \half \sum_{ij\alpha\beta} \langle J_\alpha U_j, J_\beta U_i \rangle h_{ij} h_{\alpha\beta} 
\leq Ap \smfrac{b}{q} \sum_i h_{ii}^2 + \smfrac{n}{A} \smfrac{b}{n} \sum_\alpha h_{\alpha\alpha}^2 \]
We group the remaining terms in $I$ in a way that will become clear later.
\begin{align*}
&  c \sum_{ij}(h_{ij}^2 - h_{ii} h_{jj})
+ 4c \sum_{i\alpha}(h_{i\alpha}^2 - h_{ii} h_{\alpha\alpha})
+ 4c \sum_{\alpha\beta}(h_{\alpha \beta}^2 - h_{\alpha\alpha} h_{\beta\beta}) \\
&= c \sum_{ij}(h_{ij}^2 - h_{ii} h_{jj})
+ 2c \sum_{i\alpha}(h_{i\alpha}^2 - h_{ii} h_{\alpha\alpha})
+ c \sum_{\alpha\beta}(h_{\alpha \beta}^2 - h_{\alpha\alpha} h_{\beta\beta}) \\
&\quad + 2c \sum_{i\alpha}(h_{i\alpha}^2 - h_{ii} h_{\alpha\alpha})
 + 3c \sum_{\alpha\beta}(h_{\alpha \beta}^2 - h_{\alpha\alpha} h_{\beta\beta}) \\
&\leq c \sum_{ij} (h_{ij}^2 - h_{ii} h_{jj})
+ 2c \sum_{i\alpha} (h_{i\alpha}^2 - h_{ii} h_{\alpha\alpha})
+ c \sum_{\alpha\beta} (h_{\alpha \beta}^2 - h_{\alpha\alpha} h_{\beta\beta}) \\
&\quad + 2c \sum_{i\alpha} (h_{i\alpha}^2 + \smfrac{B}{2} h_{ii}^2 + \smfrac{1}{2B} h_{\alpha\alpha}^2 )
 + 3c \sum_{\alpha\beta} h_{\alpha \beta}^2 - 3c \left( \sum_\alpha h_{\alpha\alpha} \right)^2 \\
&\leq \underline{c \sum_{ij} (h_{ij}^2 - h_{ii} h_{jj})}
+ \underline{2c \sum_{i\alpha} (h_{i\alpha}^2 - h_{ii} h_{\alpha\alpha})}
+ \underline{c \sum_{\alpha\beta} (h_{\alpha \beta}^2 - h_{\alpha\alpha} h_{\beta\beta})} \\
&\quad + 2 \smfrac{b}{q} \sum_{i\beta} h_{i\beta}^2 
       + Bm \smfrac{b}{q} \sum_i h_{ii}^2 
       + \smfrac{q}{B} \smfrac{b}{q} \sum_\alpha h_{\alpha\alpha}^2
       + 3\smfrac{b}{q} \sum_{\alpha\beta} h_{\alpha \beta}^2
\end{align*}
We continue with $II$, $III$, and $IV$.
\begin{align*}
II
&= 4 R_{IJK0} h_{I0} h_{JK} \\
&= 4 \cancel{R_{ijk0}} h_{i0} h_{jk} 
   + 4 R_{ij\gamma 0} h_{i0} h_{j\gamma} 
   + 4 R_{i\beta k0} h_{i0} h_{\beta k}
   + 4 \cancel{R_{i\beta\gamma 0}} h_{i0} h_{\beta\gamma} \\
&\quad + 4 R_{\alpha jk0} h_{\alpha 0} h_{jk} 
       + 4 \cancel{R_{\alpha j\gamma 0}} h_{\alpha 0} h_{j\gamma} 
       + 4 \cancel{R_{\alpha\beta k0}} h_{\alpha 0} h_{\beta k}
       + 4 \cancel{R_{\alpha \beta \gamma 0}} h_{\alpha 0} h_{\beta\gamma} \\
&=   4 R_{ij\alpha0} h_{i0} h_{j\alpha}
   + 4 R_{i\alpha j0} h_{i0} h_{j\alpha}
   - \cancel{4 R_{i\alpha i0} h_{\alpha0} h_{ii}} \\
&= - 6 \sqrt{c} \sum_{ij\alpha} \langle J_\alpha U_i,U_j \rangle h_{i0} h_{j\alpha} \\
&\leq \smfrac{3 \sqrt{c}}{C} \sum_{ij\alpha} |\langle J_\alpha U_i,U_j \rangle| h_{i0}^2
      + 3 \sqrt{c} C \sum_{ij\alpha} |\langle J_\alpha U_i,U_j \rangle| h_{j\alpha}^2 \\
&= \left( \smfrac{3}{C} \sqrt{\smfrac{b}{q}} \sum_{j\alpha} |\langle J_\alpha U_i,U_j \rangle| \right) \sum_i h_{i0}^2
 +  \left( 3C \sqrt{\smfrac{b}{q}} \sum_j |\langle J_\beta U_j,U_i \rangle| \right) \sum_{i\beta} h_{i\beta}^2
\end{align*}

\begin{align*}
III
&= 2 R_{I0J0} h_{I0} h_{J0} \\
&= 2 R_{i0j0} h_{i0} h_{j0} 
   + 2 \cancel{R_{i0\beta 0}} h_{i0} h_{\beta 0} 
   + 2 \cancel{R_{\alpha 0 j0}} h_{\alpha0} h_{j0}
   + 2 R_{\alpha 0\beta 0} h_{\alpha0} h_{\beta0} \\
&= 2c \, \delta_{ij} h_{i0} h_{j0} + 8c \, \delta_{\alpha\beta} h_{\alpha0} h_{\beta0} \\
&= \underline{2c \sum_i h_{i0}^2} + \underline{2c \sum_\alpha h_{\alpha0}^2}
    + 6 \smfrac{b}{q} \sum_\alpha h_{\alpha0}^2 
\end{align*}

\begin{align*}
IV
&= -2 R_{I0J0} h_{IJ} h_{00} \\
&= -2 R_{i0j0} h_{ij} h_{00} 
   - 2 \cancel{R_{i0\beta 0}} h_{i\beta} h_{00} 
   - 2 \cancel{R_{\alpha 0 j0}} h_{\alpha j} h_{00}
   - 2 R_{\alpha 0\beta 0} h_{\alpha\beta} h_{00} \\
&= - 2c \, \delta_{ij} h_{ij} - 8c \, \delta_{\alpha\beta} h_{\alpha\beta} h_{00} \\
&= - 2c \sum_i h_{ii} h_{00} - 2c \sum_\alpha h_{\alpha\alpha} h_{00} 
   - 6c \sum_\alpha h_{\alpha\alpha} h_{00} \\
&\leq - 2c \sum_i h_{ii} h_{00} - 2c \sum_\alpha h_{\alpha\alpha} h_{00} 
     +  3c \sum_\alpha \left(  \smfrac{1}{D} h_{\alpha\alpha}^2 + D h_{00}^2 \right) \\
&= - \underline{2c \sum_i h_{ii} h_{00}} - \underline{2c \sum_\alpha h_{\alpha\alpha} h_{00}} 
   + \smfrac{3}{D} \smfrac{b}{q} \sum_\alpha h_{\alpha\alpha}^2  
   + 3Dp \smfrac{b}{q} h_{00}^2
\end{align*}

With the expression for $|h|^2$ from \eqref{eq:h-norm} in mind, we combine all underlined terms to obtain a multiple of $|h|^2$ and a negative quantity.
\begin{align*}
&c \sum_{ij} (h_{ij}^2 - h_{ii} h_{jj})
+ 2c \sum_{i\alpha} (h_{i\alpha}^2 - h_{ii} h_{\alpha\alpha})
+ c \sum_{\alpha\beta} (h_{\alpha \beta}^2 - h_{\alpha\alpha} h_{\beta\beta}) \\
&\quad+ 2c \sum_i (h_{i0}^2 - h_{ii} h_{00}) 
+ 2c \sum_\alpha (h_{\alpha0}^2 - h_{\alpha\alpha} h_{00}) + c ( h_{00}^2 - h_{00}^2 )\\
&=c|h|^2 - c\left( \sum_i h_{ii} + \sum_\alpha h_{\alpha\alpha} + h_{00} \right)^2 \\
&\leq \smfrac{b}{q} |h|^2
\end{align*}

We put everything together in a way comparable to the expression for $|h|^2$,
\begin{align*}
\langle \Ro h,h \rangle 
&=   \sum_{i}       C_{i} h_{ii}^2 
 + 2 \sum_{i\beta} C_{i\beta} h_{i\beta}^2
 + 2 \sum_i C_{i0} h_{i0}^2 \\
&\qquad +   \sum_{\alpha} C_{\alpha} h_{\alpha\alpha}^2 
        + 2 \sum_{\alpha 0} C_{\alpha 0} h_{\alpha 0}^2 
        + C_{00} h_{00}^2 
\end{align*}
where
\begin{align*}
C_{i}
&:= \frac{b}{q} \Big( [3+A+B]p + 1 \Big) \\   
C_{i\beta}
&:= \frac{1}{2} \sum_{j\alpha} |\langle J_\beta U_i, J_\alpha U_j \rangle| 
  + \frac{1}{4} \sum_{j\alpha} |\langle J_\alpha U_i, J_\beta U_j \rangle|
  + \frac{3C}{2}  \sqrt{\frac{b}{q}} \sum_j |\langle J_\beta U_i,U_j \rangle| + 2 \frac{b}{q} \\
C_{i0}
&:= \frac{3}{2C} \sqrt{\frac{b}{q}} \sum_{j\alpha} |\langle J_\alpha U_i,U_j \rangle| + \frac{b}{q} \\ 
C_{\alpha}
&:= \frac{b}{q} \left( \left[ \frac{1}{A} + \frac{1}{B} \right] q + \frac{3}{D} + 4 \right) \\
C_{\alpha0}
&:= 4 \frac{b}{q} \\
C_{00}
&:= \frac{b}{q} \left( 3Dp + 1 \right)
\end{align*}

\subsection{Generalized Heisenberg algebras}\label{subsec:3step-heis}

Let $\mfr{n} := \mfr{v} \oplus \mfr{z}$ be a generalized Heisenberg algebra as in Example \ref{ex:heis-def}.  The Ricci tensors are optimal with $b = \smfrac{q}{4}$, and $-\lambda = p + \smfrac{q}{4}$.  
\[ \Ric_\mfr{n} = \left(
\begin{BMAT}{ccc.ccc}{ccc.ccc}
 &            & & &                 &  \\
 & -\frac{p}{2} \, \id_\mfr{v}     & & &                 &  \\
 &            & & &                 &  \\
 &            & & &                 &  \\
 &            & & & \frac{q}{4} \, \id_\mfr{z} &  \\
 &            & & &                 &  \\
\end{BMAT} \right)
\qquad 
D = \left(
\begin{BMAT}{ccc.ccc}{ccc.ccc}
 &            & & &                 &  \\
 & \left(\frac{p}{2}+ \frac{q}{4} \right) \, \id_\mfr{v}     & & &                 &  \\
 &            & & &                 &  \\
 &            & & &                 &  \\
 &            & & & \left(\frac{p}{2} + q\right) \, \id_\mfr{z} &  \\
 &            & & &                 &  \\
\end{BMAT} \right)\]
We also have $c=1/4$.  

Recall that complex hyperbolic space is a solvable Lie group, and its Lie algebra is a one-dimensional extension of a classical Heisenberg algebra.  Similarly, one can introduce a specific one-dimensional extension of any generalized Heisenberg algebra to create a solvable Lie algebra.  It turns out that these algebras always carry Einstein metrics, and the corresponding simply connected Lie groups are are called \textit{Damek-Ricci spaces}; see \cite{Berndt1995}.

We will prove linear stability of all Damek-Ricci spaces.  The case of complex hyperbolic space ($p=1$) was covered in \cite{Wu2013}.  We will assume that $p \geq 2$, and additionally that $p \leq q/4$.  This excludes a finite number of cases, whose stability can be checked separately.  To proceed, we need information on the basis on any Damek-Ricci space.  Recall that a basis $\{U_i\}$ of a Lie algebra is \textit{nice} if $[U_i,U_j]$ is always a scalar multiple of some element in the basis, and two different brackets $[U_i,U_j]$, $[U_k,U_\ell]$ can be a nonzero multiple of the same $U_r$ only if $\{i,j\}$ and $\{k,\ell\}$ are disjoint.  See \cite{LauretWill2013}.

\begin{lem}
The standard basis for the any Damek-Ricci space is nice, and all structure constants are $0$, $1$, or $-1$.
\end{lem}
 
\begin{proof}
Recall that when $p \not \equiv 3\pmod{4}$, there is one irreducible Clifford module $\mfr{d}$ over $\mrm{Cl}(\mfr{z},\eta)$.  Every Clifford module $\mfr{v}$ over $\mrm{Cl}(\mfr{z},\eta)$ is isomorphic to a sum $\mfr{v} = \oplus^k \mfr{d}$.  If $p  \equiv 3\pmod{4}$, then there are two non-equivalent irreducible Clifford modules over $\mrm{Cl}(\mfr{z},\eta)$, $\mfr{d}_1$ and $\mfr{d}_1$.  Every Clifford module $\mfr{v}$ over $\mrm{Cl}(\mfr{z},\eta)$ is isomorphic to a sum $\mfr{v} = (\oplus^k \mfr{d}_1) \oplus (\oplus^\ell \mfr{d}_2)$.
 
Also recall that a basis for the Clifford algebra $\mrm{Cl}(\mfr{z},\eta)$ is
\[ \{ J_{\alpha_1} \cdots J_{\alpha_k} \,|\, 1 \leq \alpha_1 < \cdots < \alpha_k \leq p, 1 \leq k \leq q \}. \]
As the structure constants of $\mfr{n}$ are determined from the $J$-map by the relation
\[ c_{ij}^\alpha = (J_\alpha)_i^j, \]
we simply need to describe the matrices for the $J_\alpha$ with respect to the above basis of $\mrm{Cl}(\mfr{z},\eta)$.  To do this, we examine the action on the basis by each $J_\alpha$, which is Clifford multiplication by $J_\alpha$.  By examining the possible products $J_\alpha J_{\alpha_1} \cdots J_{\alpha_k}$, it is easy to see that $J_\alpha$ acts as a permutation of the basis, except possibly with signs introduced from the Clifford relations:
\[ J_\alpha J_\beta + J_\beta J_\alpha = -2 \delta_{\alpha \beta} \, 1. \]
Moreover, each $J_\alpha$ is skew-symmetric, so this means the matrix for each $J_\alpha$ is a skew-symmetric permutation matrix.  
 
Next we claim that for any fixed $i, j$, there can be at most one $\alpha$ such that $c_{ij}^\alpha$ is non-zero.  Indeed, suppose that $c_{ij}^\alpha = \pm 1$ and let $\beta$ be any other index besides $\alpha$.  By the Clifford relations, we have
\begin{align*}
0
&= \langle J_\alpha J_\beta U_i,U_j \rangle + \langle J_\beta J_\alpha U_i, U_j \rangle \\
&= \sum_k \Big( (J_\beta)_i^k (J_\alpha)_k^j + (J_\alpha)_i^k (J_\beta)_k^j \Big) \\
&= \sum_k \Big( \pm c_{ik}^\beta \delta_{jk}  \pm \delta _{ik} c_{kj}^\beta \Big) \\
&= \pm 2 c_{ij}^\beta,
\end{align*}
and so $c_{ij}^\beta = 0$ for all $\beta \ne \alpha$.  (Of course, it could happen that there are $i$ and $j$ such that $c_{ij}^\alpha$ is always zero.)  

Finally, if the representation is not irreducible, then it is a sum of the irreducible ones just described, and all the properties still hold.
 \end{proof}

Using a nice basis, it is not hard to verify the following estimates.
\begin{align*}
\sum_{j\alpha} |\langle J_\alpha U_j, J_\beta U_i \rangle| 
&\leq p 
& & & 
\sum_{j\alpha} |\langle J_\alpha U_i, J_\beta U_j \rangle| 
&\leq p \\
\sum_{j\alpha} |\langle J_\alpha U_i, U_j \rangle| 
&\leq p 
& & & 
\sum_j |\langle J_\alpha U_i, U_j \rangle| 
&\leq 1 
\end{align*}
If we set $A=B=1/2$ and $C=D=E=1$, the coefficients in $\langle \Ro h,h \rangle$ become the following.
\begin{align*}
C_{i}
&\leq \frac{3}{4} p + \frac{1}{16}\left( \frac{1}{A} + \frac{1}{B} \right) q + \frac{1}{4} 
    = \frac{1}{4} q + \frac{3}{4} p + \frac{1}{4} \\
C_{i\beta}
&\leq \frac{3}{4} p + \frac{3C}{4} + \frac{1}{2} 
    = \frac{3}{4} p + \frac{5}{4} \\
C_{i0}
&\leq \frac{3}{4C} p + \frac{1}{4} 
    = \frac{3}{4} p + \frac{1}{4} \\
C_{\alpha}
&\leq \frac{A+B}{4} q + \frac{3}{4D} + 1 
    = \frac{1}{4} q + \frac{7}{4} \\ 
C_{\alpha0}
&= 1 \\
C_{00}
&\leq \frac{3D}{4} p + 1
    = \frac{3}{4} p + 1
\end{align*}
We use that $q \geq 4$ to see rewrite
\[
C_{i\beta}
\leq \frac{3}{4} p + \frac{5}{4}
= \frac{3}{4} p + \frac{1}{4} + 1 
\leq \frac{1}{4} q + \frac{3}{4} p + \frac{1}{4}.
\]
and 
\[ 
C_{00} =\frac{3}{4} p + 1 \leq \frac{3}{4} p + \frac{1}{4} q < \frac{1}{4} q + p.
\]
It is easy to see that $\frac{3}{4} p + \frac{1}{4} < p$ when $p \geq 1$, so we have that each coefficient is less than $p+q/4$.  See Table \ref{table:damek-ricci} for the remaining cases.  The notation $\mfr{h}_{p,q}$ refers to the generalized Heisenberg algebra with $\dim \mfr{z}=p$ and $\dim \mfr{v}=q$.

\begin{table}[t]
\begin{tabular}{|c|cccc|}
\hline
$\mfr{h}_{p,q}$ & $A$ & $B$ & $C$ & $D$ \vv \\
\hline
$\mfr{h}_{3,4}$ & $\frac{13}{8}$ & $\frac{3}{4}$ & $\frac{17}{15}$ & $\frac{5}{4}$ \vv \\
$\mfr{h}_{3,8}$ & $\frac{1}{2}$ & $1$ & $\frac{33}{19}$ & $1$ \vv \\ \vv
$\mfr{h}_{4,8}$ & $\frac{7}{4}$ & $\frac{15}{32}$ & $\frac{133}{69}$ & $\frac{3}{2}$ \\
$\mfr{h}_{5,8}$ & $1$ & $1$ & $\frac{19}{9}$ & $1$ \vv \\
$\mfr{h}_{5,16}$ & $\frac{51}{32}$ & $\frac{19}{64}$ & $\frac{17}{21}$ & $2$ \vv \\
$\mfr{h}_{6,8}$ & $\frac{21}{8}$ & $\frac{9}{16}$ & $\frac{71}{31}$ & $\frac{3}{2}$ \vv \\
$\mfr{h}_{6,16}$ & $\frac{7}{4}$ & $\frac{11}{32}$ & $\frac{139}{39}$ & $\frac{3}{2}$ \vv \\
$\mfr{h}_{7,8}$ & $3$ & $\frac{5}{8}$ & $\frac{37}{15}$ & $\frac{3}{2}$ \vv \\
$\mfr{h}_{7,16}$ & $1$ & $1$ & $\frac{161}{43}$ & $1$ \vv \\
$\mfr{h}_{7,24}$ & $\frac{107}{64}$ & $\frac{139}{512}$ & $\frac{1257}{51}$ & $\frac{9}{4}$ \vv \\
$\mfr{h}_{8,16}$ & $\frac{141}{64}$ & $\frac{53}{128}$ & $\frac{553}{141}$ & $\frac{3}{2}$ \vv \\
$\mfr{h}_{9,32}$ & $\frac{27}{16}$ & $\frac{133}{512}$ & $\frac{449}{67}$ & $2$ \vv \\
\hline 
\end{tabular}
\caption{Constants that give linear stability of Einstein metrics on ``low-dimensional'' Damek-Ricci spaces.}
\label{table:damek-ricci}
\end{table}

\begin{rem}
Setting aside the results of the previous section, we note that we can use Proposition \ref{prop:ext-stab} to show that the generalized Heisenberg nilsolitons are linearly stable. 
\end{rem}

\subsection{Free two-step nilpotent algebras}\label{subsec:3step-free}

Let $\mfr{n}$ be the free 2-step nilpotent Lie algebra on $q$ generators, $\mfr{F}_2(q) = \R^q \oplus \mfr{so}(q,\R)$, as in Example \ref{ex:free-2step-def}.  By results of Eberlein \cite{Eberlein2008}, the Ricci tensor is optimal with $b=\half$ and $p= \half q(q-1)= \dim \mfr{so}(n)$, and the soliton derivation is easily computed:
\[ \Ric_\mfr{n} = \frac{1}{2} \left(
\begin{BMAT}{ccc.ccc}{ccc.ccc}
 &                        & & &                 &  \\
 & -(q-1) \id_\mfr{v}     & & &                 &  \\
 &                        & & &                 &  \\
 &                        & & &                 &  \\
 &                        & & & \id_\mfr{z}     &  \\
 &                        & & &                 &  \\
\end{BMAT} \right),
\quad
D = \frac{1}{2} q \left(
\begin{BMAT}{ccc.ccc}{ccc.ccc}
 &             & & &               &  \\
 & \id_\mfr{v} & & &               &  \\
 &             & & &               &  \\
 &             & & &               &  \\
 &             & & & 2 \id_\mfr{z} &  \\
 &             & & &               &  \\
\end{BMAT} \right)\]
and $-\lambda = q-\half$.  

We use the following orthonormal basis for $\mfr{n}$
\[ \{U_1,\dots,U_q,Z_1,\dots,Z_p\}, \]
where the $U_i$ are an orthonormal basis of $\R^q$ and the $Z_\alpha$ are the usual basis for $\mfr{so}(q)$, properly normalized.  Since each $Z_\alpha$ is a $q\times q$ matrix, write $\alpha = \{k,\ell\}$ for $1 \leq k<\ell \leq q$.  The result is that given any two $1\leq i<j\leq q$ there is precisely one $1 \leq \alpha \leq p$ such that
\[ [U_i,U_j] = Z_\alpha. \]
Similarly, given $\alpha$ and $i$, the $J$-maps satisfy
\[ J_\alpha U_i = \begin{cases}
0   & \text{if } i \not\in \alpha \\
U_j & \text{if } \alpha = \{i,j\}
\end{cases}. \]
From these properties, it is easy to derive the following estimates.
\begin{align*}
\sum_{j\alpha} |\langle J_\alpha U_j, J_\beta U_i \rangle| 
&\leq q-1 
& & &
\sum_{j\alpha} |\langle J_\alpha U_i, J_\beta U_j \rangle| 
&\leq 1  \\
\sum_{j\alpha} |\langle J_\alpha U_i, U_j \rangle| 
&\leq q-1 
& & &
\sum_j |\langle J_\alpha U_i, U_j \rangle| 
&\leq 1 
\end{align*}
Using these estimates, and setting 
\[ A = \frac{1}{2}, \quad B = \frac{25}{64}, \quad C = \frac{7}{4}, \quad D = \frac{4}{3}, \]
the coefficients in $\langle \Ro h,h \rangle$ become the following.
\begin{align*}
C_{i}
&\leq \frac{1}{4}(3+A+B)(q-1) + \frac{1}{2q} 
    = -\frac{249}{256}+\frac{1}{2q}+\frac{249q}{256} \\
C_{i\beta}
&\leq \frac{1}{2} (q-1) + \frac{1}{4} + \frac{3C}{2 \sqrt{2q}} + \frac{1}{q} 
    = -\frac{1}{4}+\frac{1}{q}+\frac{21}{8 \sqrt{2} \sqrt{q}}+\frac{q}{2} \\
C_{i0}
&\leq \frac{3}{2C \sqrt{2q}} (q-1) + \frac{1}{2q} 
    = -\frac{3 \sqrt{2}}{7 \sqrt{q}}+\frac{3 \sqrt{2} \sqrt{n}}{7} \\
C_{\alpha}
&\leq \frac{1}{2} \left( \frac{1}{A} + \frac{1}{B} \right) + \frac{3}{2D q} + \frac{2}{q}
    = \frac{57}{25}+\frac{25}{8 q} \\ 
C_{\alpha0}
&\leq \frac{2}{q} \\
C_{00}
&\leq \frac{3D}{4} (q-1) + \frac{1}{2q}
    = -1+\frac{1}{2 q} + q
 \end{align*}
It is not hard to see that when $q \geq 4$, each is less than $q-\half$.  

The case $q=2$ is already covered, since $\mfr{n} = \R^2 \oplus \mfr{so}(2,\R) \cong \mfr{nil}^3$.  The case $q=3$ is already covered, since $\mfr{n} = \R^3 \oplus \mfr{so}(3,\R)$ corresponds with the 24th six-dimensional algebra from Table \ref{table:solitons}.  This shows strict linear stability of the Einstein metrics.  

\begin{rem}
As before, we note that stability of the nilsolitons can be obtained by Proposition \ref{prop:ext-stab}. 
\end{rem}

\subsection{Negatively curved spaces}

Here we give a brief account of some results on certain negatively-curved homogeneous spaces, and explain their relevance to our stability results, namely Proposition \ref{prop:sec-stab}.

In \cite{EberleinHeber1996}, Eberlein and Heber consider homogeneous spaces with negatively quarter-pinched sectional curvature.  It is well-known that any homogeneous space with non-positive curvature can be represented as a simply connected solvable Lie group with a left-invariant metric.  On the Lie algebra level, the authors show that, in fact, any solvable $(\mfr{s},\langle \cdot,\cdot \rangle)$ with negatively quarter-pinched curvature is isometric to an algebra that is closely related to those that we have studied above.  Namely, the algebra is either two-step solvable, three-step solvable with optimal Ricci tensor on the nilradical, or a certain combination of these two types (called an amalgamated product) \cite[Theorem 5.1]{EberleinHeber1996}.  Whenever the metric on such a space is Einstein or a solvsoliton, Proposition \ref{prop:sec-stab} will ensure linear stability.

The authors also give a characterization of rank-one symmetric spaces.  Namely, let $(\mfr{s},\langle \cdot,\cdot \rangle)$ be a solvable Lie algebra with negatively quarter-pinched sectional curvature.  Then the space is rank-one symmetric if and only if it is Einstein \cite[Theorem 7.1]{EberleinHeber1996}.  This narrows the range of spaces for which Proposition \ref{prop:sec-stab} is useful, at least in the quarter-pinched homogeneous Einstein case.  For example, symmetric spaces of non-compact type are shown by Bamler to be stable under curvature-normalized Ricci flow \cite{Bamler2010-sym}.  

If we do not assume quarter-pinched curvature, the Proposition is still useful.  For example, Heber has obtained large families of negatively-curved Einstein metrics that are deformations of certain hyperbolic spaces.  Namely, for each $m \geq 2$, there is an $(8m^2-6m-8)$-dimensional family of negatively-curved Einstein metrics containing the quaternionic hyperbolic space $\mbb{H}H^{m+1}$.  Also,
there is an $84$-dimensional family of negatively-curved Einstein metrics containing the Cayley hyperbolic plane $\mbb{C}aH^2$; see \cite[Theorem J]{Heber1998}.  These metrics are all stable by Proposition \ref{prop:sec-stab}.  (Note that both $\mbb{H} H^{m+1}$ and $\mbb{C}a H^2$ are Damek-Ricci spaces; see \cite[Subsection 4.1.9]{Berndt1995}.)

In the case of only non-positive curvature, Proposition \ref{prop:sec-stab} can only give weak linear stability for Einstein metrics.  For example, the Einstein metrics in Section \ref{sec:3step} all have non-positive curvature, which is why we instead appeal to Corollary \ref{cor:alg-stab-cond}.  For solvsolitons, the derivation only needs to be positive, so the situation is more promising, as we saw in Section \ref{sec:abelian}.

\section{Low-dimensional examples}\label{sec:low-dim}

\begin{table}[t]
\begin{tabular}{|ccccccc|ccc|}
\hline
 \multicolumn{7}{|c|}{Nilsoliton}     & \multicolumn{3}{c|}{Solvable extension} \\    
dim  &  number &  step  &  $\lambda$ &  $\tr D$  &  max $Q$ &  $\overset{?}{<} \half \tr D$ &  dim  &  max $\Ro$ &  $\overset{?}{<} -\lambda$ \\
\hline
1 & 1 & 1 & -1   & 1     & 0     & \checkmark & 2 & 1     & (weak) \\
2 & 1 & 1 & -1   & 2     & 0     & \checkmark & 3 & 0.5   & \checkmark \\
3 & 1 & 2 & -1.5 & 4     & 0.569 & \checkmark & 4 & 1     & \checkmark \\
3 & 2 & 1 &   -1 & 3     & 0     & \checkmark & 4 & 0.333 & \checkmark \\
4 & 1 & 2 & -1.5 & 5     & 0.783 & \checkmark & 5 & 1.355 & \checkmark \\
4 & 2 & 2 & -1.5 & 5.5   & 0.569 & \checkmark & 5 & 0.932 & \checkmark \\
4 & 3 & 1 & -1   & 4     & 0     & \checkmark & 5 & 0.25  & \checkmark \\
5 & 1 & 4 & -2   & 8.333 & 1.174 & \checkmark & 6 & 1.768 & \checkmark \\
5 & 2 & 4 & -5.5 & 22.5  & 2.601 & \checkmark & 6 & 4.322 & \checkmark \\
5 & 3 & 3 & -3.5 & 15    & 1.751 & \checkmark & 6 & 2.79  & \checkmark \\
5 & 4 & 2 & -2   & 9     & 1.106 & \checkmark & 6 & 1     & \checkmark \\
5 & 5 & 3 & -6   & 25    & 2.026 & \checkmark & 6 & 4.122 & \checkmark \\
5 & 6 & 2 & -2   & 9     & 0.58  & \checkmark & 6 & 1.249 & \checkmark \\
5 & 7 & 2 & -1.5 & 7     & 0.569 & \checkmark & 6 & 0.893 & \checkmark \\
5 & 8 & 3 & -1.5 & 6.5   & 0.783 & \checkmark & 6 & 1.306 & \checkmark \\
5 & 9 & 1 & -1   & 5     & 0     & \checkmark & 6 & 0.2   & \checkmark \\
\hline 
\end{tabular}
\caption{Linear stability of nilsolitons of dimension five or less and corresponding 1-dimensional solvable Einstein extensions}
\label{table:solitons}
\end{table}

\begin{table}
\begin{tabular}{|ccccccc|ccc|}
\hline
 \multicolumn{7}{|c|}{Nilsoliton}     & \multicolumn{3}{c|}{Solvable extension} \\    
dim  &  number &  step  &  $\lambda$ &  $\tr D$  &  max $Q$ &  $\overset{?}{<} \half \tr D$ &  dim  &  max $\Ro$ &  $\overset{?}{<} -\lambda$ \\
\hline
6 & 1  & 5 & -26   & 121    & 14.514 & \checkmark & 7 & 20.412 & \checkmark \\
6 & 2  & 5 & -2    & 9.333  & 1.158  & \checkmark & 7 & 1.725  & \checkmark \\
6 & 3  & 5 & -11   & 56     & 6.149  & \checkmark & 7 & 8.83   & \checkmark \\
6 & 4  & 5 & -71.5 & 346.5  & 35.169 & \checkmark & 7 & 58.873 & \checkmark \\
6 & 5  & 5 & -17   & 84.5   & 7.831  & \checkmark & 7 & 12.479 & \checkmark \\
6 & 6  & 4 & -2    & 10     & 0.806  & \checkmark & 7 & 1.412  & \checkmark \\
6 & 7  & 4 & -10   & 49     & 4.564  & \checkmark & 7 & 7.71   & \checkmark \\
6 & 8  & 4 & -10   & 49     & 4.881  & \checkmark & 7 & 6.77   & \checkmark \\
6 & 9  & 4 & -2.5  & 12.25  & 1.369  & \checkmark & 7 & 2.086  & \checkmark \\
6 & 10 & 4 & -2.5  & 12.667 & 1.355  & \checkmark & 7 & 2.077  & \checkmark \\
6 & 11 & 4 & -1.44 & 7.29   & 0.732  & \checkmark & 7 & 1.166  & \checkmark \\
6 & 12 & 4 & -5.5  & 28     & 2.601  & \checkmark & 7 & 4.196  & \checkmark \\
6 & 13 & 4 & -6    & 31     & 3.523  & \checkmark & 7 & 5.212  & \checkmark \\
6 & 14 & 3 & -5.5  & 28     & 2.836  & \checkmark & 7 & 4.165  & \checkmark \\
6 & 15 & 3 & -2    & 10     & 1.131  & \checkmark & 7 & 1.494  & \checkmark \\
6 & 16 & 3 & -2    & 10     & 1.144  & \checkmark & 7 & 1.709  & \checkmark \\
6 & 17 & 3 & -6    & 31     & 2.026  & \checkmark & 7 & 4.02   & \checkmark \\
6 & 18 & 3 & -3.5  & 18     & 1.217  & \checkmark & 7 & 2.022  & \checkmark \\
6 & 19 & 3 & -6.5  & 33.5   & 2.336  & \checkmark & 7 & 4.464  & \checkmark \\
6 & 20 & 3 & -7    & 36     & 2.454  & \checkmark & 7 & 5.474  & \checkmark \\
6 & 21 & 3 & -4    & 21     & 1.755  & \checkmark & 7 & 2.69   & \checkmark \\
6 & 22 & 3 & -1.5  & 7.75   & 0.646  & \checkmark & 7 & 1.217  & \checkmark \\
6 & 23 & 3 & -3.5  & 18.5   & 1.537  & \checkmark & 7 & 2.67   & \checkmark \\
6 & 24 & 2 & -2.5  & 13.5   & 0.581  & \checkmark & 7 & 1.071  & \checkmark \\
6 & 25 & 3 & -6    & 31     & 3.412  & \checkmark & 7 & 4.678  & \checkmark \\
6 & 26 & 3 & -1.5  & 8      & 0.783  & \checkmark & 7 & 1.278  & \checkmark \\
6 & 27 & 3 & -3.5  & 18.5   & 1.751  & \checkmark & 7 & 2.705  & \checkmark \\
6 & 28 & 2 & -3    & 16     & 1.137  & \checkmark & 7 & 1.75   & \checkmark \\
6 & 29 & 2 & -2.5  & 13.5   & 1.094  & \checkmark & 7 & 1.418  & \checkmark \\
6 & 30 & 2 & -1.5  & 8      & 0.569  & \checkmark & 7 & 0.875  & \checkmark \\
6 & 31 & 2 & -2    & 11     & 0.580  & \checkmark & 7 & 1.066  & \checkmark \\
6 & 32 & 2 & -2    & 11     & 1.106  & \checkmark & 7 & 0.955  & \checkmark \\
6 & 33 & 2 & -1.5  & 8.5    & 0.569  & \checkmark & 7 & 0.868  & \checkmark \\
6 & 34 & 1 & -1    & 6      & 0      & \checkmark & 7 & 0.167  & \checkmark \\
\hline 
\end{tabular}
\caption{Linear stability of nilsolitons of dimension six and corresponding 1-dimensional solvable Einstein extensions}
\label{table:solitons2}
\end{table}

Nilpotent Lie groups admitting nilsoliton metrics have been completely classified in low dimensions; a list of nilsolitons in dimensions up to six appears in \cite{Will2011} (progress has also been made in dimensions 7 and 8 \cite{FernandezCulma2012,Arroyo2011,KadiogluPayne2013}).  In this section, we show that all such nilsolitons are linearly stable, and each has a one-dimensional Einstein extension (described following Theorem \ref{thm:lauret}) that is linearly stable as well.  See Tables \ref{table:solitons} and \ref{table:solitons2} for a list of these nilsolitons (as described in \cite{Will2011}), which includes the dimension, step, soliton constant, trace of the soliton derivation, and largest eigenvalue of the operator $Q$.  The tables also includes the largest eigenvalue of the operator $\Ro$ for each 1-dimensional solvable Einstein extension.

The method to determine stability is to use Corollary \ref{cor:alg-stab-cond} and see that the required bounds on the relevant operators are indeed satisfied.  Using a computer algebra system, it is possible to input the structure constants for a Lie algebra and from them calculate all related geometric quantities.  In particular, with respect to an orthonormal basis for the space of symmetric 2-tensors, one can describe $\Ro$ and $Q$ as matrices and calculate their respective eigenvalues.  While the eigenvalues in Tables \ref{table:solitons} and \ref{table:solitons2} are only approximate, they can be described to great precision as roots of characteristic polynomials.

\begin{prop}\label{prop:low-dim}
Each nilsoliton metric on a nilpotent Lie algebra of dimension six or less is strictly linearly stable.  Each Einstein metric on the corresponding 1-dimensional extension is strictly linearly stable.
\end{prop}

We finish examples, including two continuous families of stable metrics.

\begin{ex}\label{ex:mu11}
Here we consider the 6-dimensional algebra listed as \#11 in Table \ref{table:solitons2}.  This does appear in the classifications in \cite{Will2003,Will2011}, but the explicit nilsoliton is incorrect; see \cite[Remark 5.3]{LauretWill2013}.  What follows is an alternate description.

With respect to an orthonormal basis $\{X_1,\dots,X_6\}$, the non-zero structure constants are 
\begin{align*}
[X_1,X_2] &= X_4  & & &  [X_1,X_4] &= X_5  & & &  [X_1,X_5] &= X_6 \\
[X_2,X_3] &= X_6  & & &  [X_2,X_4] &= X_6  & & &  
\end{align*}
Represent this structure as $\mu_{11} \in \mcl{V}_6$, where $\mcl{V}_6 \subset \wedge^2 (\R^6)^* \otimes \R^6$ is the variety of 6-dimensional nilpotent bracket structures.  The Ricci tensor is
\small
\[ \Ric = -\frac{1}{2} 
\begin{pmatrix}
3 &   &   &   &   & \\
  & 3 &   &   &   & \\
  &   & 1 & 1 &   & \\
  &   & 1 & 1 &   & \\
  &   &   &   & 0 & \\
  &   &   &   &   & -3
\end{pmatrix}. \]
\normalsize
Consider
\small
\[ G := 
\left(
\begin{array}{cccccc}
 \frac{\sqrt{10}}{3} &   &   &   &   &   \\
   & \sqrt{\frac{5}{3}} &   &   &   &   \\
   &   & 0  & \sqrt{\frac{2}{3}} &   &   \\
   &   & 1 & \frac{1}{\sqrt{6}} &   &   \\
   &   &   &   & 1 &   \\
   &   &   &   &   & \sqrt{\frac{3}{5}}
\end{array}
\right) \in \GL(6,\R), \]
\normalsize
so that under the action of $\GL(6,\R)$ on $\mcl{V}_6$ given by 
\[ G \cdot \mu(X,Y) = G \mu(G^{-1}X, G^{-1}Y), \]
we have an isomorphic Lie algebra $G \cdot \mu_{11}$ whose non-zero structure constants are
\begin{align*}
[X_1,X_2] &= \frac{3}{5} \sqrt{\frac{3}{2}} X_3  + \frac{3}{10} X_4 & & & 
[X_1,X_3] &= \frac{3}{\sqrt{10}} X_5 & & &
[X_1,X_5] &= \frac{3}{5} \sqrt{\frac{3}{2}} X_6 \\
[X_2,X_3] &= \frac{3}{10} X_6 & & &
[X_2,X_4] &= \frac{3}{5} \sqrt{\frac{3}{2}} X_6 
\end{align*}
This new bracket representation results in a diagonal Ricci tensor, and one can check that the soliton derivation and constant are as follows:
\[ \begin{aligned}
\Ric &= \frac{1}{200} \mrm{diag}(-207,-126,-45,-45,36,117) \\
D &= \frac{1}{200} \mrm{diag}(81,162,243,243,324,405) \\
\lambda &= -\frac{36}{25}
\end{aligned} \]
The maximum eigenvalues of Q (and of $\Ro$ for the one-dimensional Einstein extension) can now be computed, and appear in Table \ref{table:solitons2}.
\end{ex}

\begin{ex}\label{ex:jorgescurve}
Lauret has produced a curve of pairwise non-isometric $8$-dimensional Einstein solvmanifolds \cite{Lauret2002} that depend on a parameter $0<t<1$.  With respect to an orthonormal basis $\{X_1,\dots,X_7\}$, the corresponding codimension-one nilradical has Lie bracket relations
\begin{align*}
[X_1,X_2] &= (1-t)^{1/2} X_3  & & &  [X_2,X_3] &= X_5 \\
[X_1,X_3] &= X_4              & & &  [X_2,X_4] &= X_6 \\
[X_1,X_4] &= t^{1/2} X_4      & & &  [X_2,X_5] &= t^{1/2} X_7 \\
[X_1,X_5] &= X_6              & & &  [X_3,X_4] &= (1-t)^{1/2} X_7 \\
[X_1,X_6] &= X_7              & & &  &
\end{align*}
This is a nilsoliton:
\[ \begin{aligned}
\Ric &= - \frac{1}{2} \mrm{diag}(4,3,2,1,0,-1,-2) \\
D &=  \frac{1}{2} \mrm{diag}(1,2,3,4,5,6,7) \\
\lambda &= -\frac{5}{2}
\end{aligned} \]
Note that these quantities do not depend on $t$.  We can form the Einstein extensions, and compute the curvatures numerically and see that the eigenvalues of $\Ro_t$ are less than $-\lambda = \smfrac{5}{2}$ for each $0<t<1$, so we have strict linear stability.  See Figure \ref{fig:curve}.  Linear stability of the nilsolitons follows from Proposition \ref{prop:ext-stab}, since 
\[ 3.5 = \max\{d_i\} < \frac{1}{2+\sqrt{2}} \tr D = \frac{14}{2 + \sqrt{2}} \cong 4.1. \]
\end{ex}

\begin{figure}[t]
\includegraphics[width=0.8\textwidth]{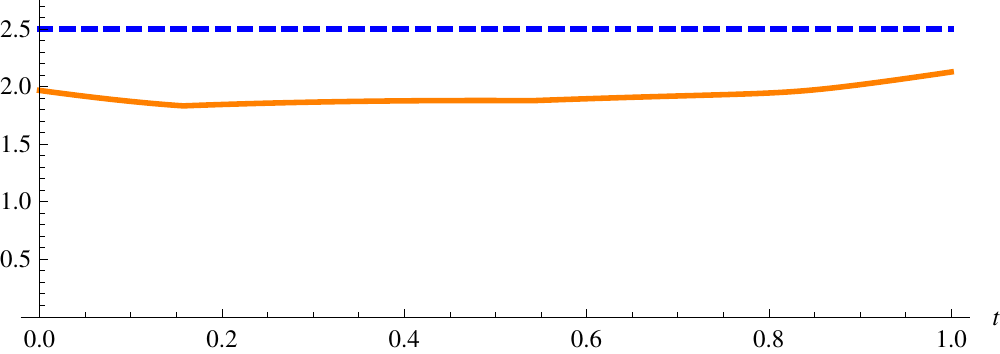}
\caption{Stability of the curve of Einstein metrics in Example \ref{ex:jorgescurve}.  The solid line is $\max \langle \Ro_t h,h \rangle/|h|^2$ and the dashed line is $-\lambda$.}
\label{fig:curve}
\end{figure}

\begin{ex}\label{ex:nil3-solv-ext}
Here we consider a family of non-Einstein solvsolitons.  Recall that for a given nilsoliton there is a moduli space of corresponding solvsoliton extensions, and Will has described all such moduli spaces in low dimensions \cite{Will2011}.  We consider the 1-parameter space of non-Einstein solvsolitons constructed from $\mfr{nil}^3$, found in \cite[Section 4.1]{Will2011}.  For $t \in (-1/\sqrt{2},1/\sqrt{2})$, we have non-zero bracket relations
\begin{align*}
[X_1,X_2] = 3 \quad 
[A,X_1] = t X_1 \quad 
[A,X_2] = \sqrt{1-t^2} X_2 \quad
[A,X_3] = (t+\sqrt{1-t^2}) X_3 
\end{align*}
and the solvsoliton metrics are given by
\[ \langle X_i,X_j \rangle = \delta_{ij} \quad \langle A,A \rangle = \frac{4}{3}(1 + t \sqrt{1-t^2}). \]
One can compute all curvatures and see that for each $t$, $Q_t(h) < \half \tr D_t |h|^2$ for all non-zero $h$, so all solvsolitons are stable.  See Figure \ref{fig:nil3-moduli}.
\end{ex}

\begin{figure}[t]
\includegraphics[width=0.8\textwidth]{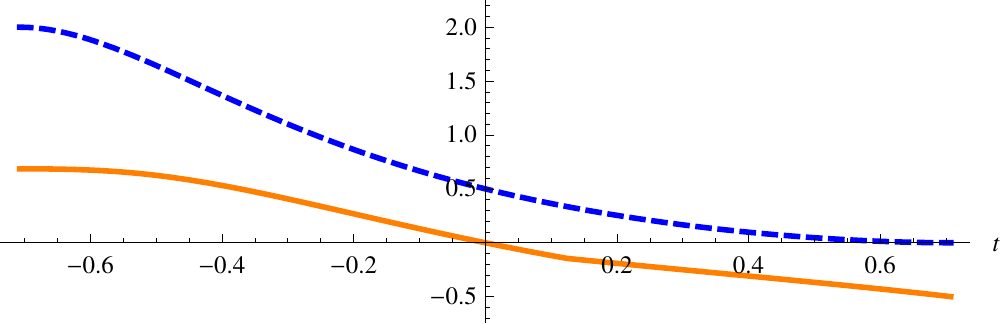}
\caption{Stability of the curve of solvsolitons in Example \ref{ex:nil3-solv-ext}.  The solid line is $\max Q_t(h)/|h|^2$ and the dashed line is $\half \tr D_t$.}
\label{fig:nil3-moduli}
\end{figure}

%%% -------------------------------------------------------------------
%%% bibliography
%%%-------------------------------------------------------------------

%\bibliography{../refs}


\end{document}